\documentclass{daj}

\newcommand{\N}{\mathbb N}
\newcommand{\Z}{\mathbb Z}
\newcommand{\Q}{\mathbb Q}
\newcommand{\Orb}{\mathcal O}
\usepackage{color}
\usepackage{pgf}
\usepackage{tikz}
\usetikzlibrary{arrows,automata}
\usepackage[latin1]{inputenc}
\usepackage{verbatim}

\usepackage{pgf}
\usepackage{tikz}
\usetikzlibrary{arrows,automata}
\usepackage[latin1]{inputenc}
\usepackage{verbatim}
\usepackage{amsthm,amsmath,amsfonts,amssymb}

\newtheorem{theorem}{Theorem}[section]
\newtheorem{lemma}[theorem]{Lemma}
\newtheorem{corollary}[theorem]{Corollary}

\newtheorem{proposition}[theorem]{Proposition}

\newtheorem{example}[theorem]{Example}

\theoremstyle{definition}
\newtheorem{definition}{Definition}

\dajAUTHORdetails{%
  title = {Computing Automorphism Groups of Shifts using Atypical 
Equivalence Classes}, 
  author = {Ethan M. Coven, Anthony Quas,  and Reem Yassawi},
  plaintextauthor = {Ethan M. Coven, Anthony Quas,  and Reem Yassawi},
  runningtitle = {Computing Automorphism Groups}, 
  runningauthor = {  Ethan M. Coven, Anthony Quas,  and Reem Yassawi    },
  keywords = {substitution dynamical systems, endomorphisms},
}   

\dajEDITORdetails{%
   year={2016},
   number={3},
   received={14 October 2015},   
   revised={4 January 2016},    
   published={28 February 2016},  
   doi={10.19086/da.611},       
}   

\begin{document}

\begin{frontmatter}[classification=text]


\author[ethan]{Ethan M. Coven}
\author[anthony]{Anthony Quas\thanks{Supported by NSERC }}
\author[reem]{Reem Yassawi}

\begin{abstract}
We study the automorphism group of an infinite minimal shift $(X,\sigma)$
such that the complexity difference function, $p(n+1)-p(n)$, is bounded. 
We give some new bounds on $\mbox{Aut}(X,\sigma)/\langle\sigma\rangle$
and also study the one-sided case. 
For a class of Toeplitz shifts, including the class of  
shifts defined by constant-length primitive substitutions with 
a coincidence, and with height one, 
we show that the two-sided automorphism group is a cyclic group. 
We next focus on shifts generated by  primitive constant-length 
substitutions. For these shifts,
we give an algorithm that computes their two-sided 
automorphism group. Finally we show that with the same techniques, 
we are able to   compute the set of 
conjugacies between two such shifts.
\end{abstract}
\end{frontmatter}

\section{Introduction}

In this article we study the automorphism  groups $\mbox{Aut}(X,\sigma)$ of 
some ``small" shifts $(X,\sigma)$. All terms are defined as they are needed 
in Sections \ref{sec:sublin} to \ref{sec:coincidences}. 
 
In the brief Section \ref{sec:sublin},
we show that for shifts with sublinear complexity, there are bounds on the 
orders of $\mbox{Aut}(X,\sigma)$  for one-sided shifts and 
$\mbox{Aut}(X,\sigma)/\langle \sigma\rangle$ for two-sided shifts. 

In \cite{coven_automorphism}, Coven explicitly describes 
the endomorphism monoid of 
any non-trivial two-sided constant-length substitution shift  
$(X,\sigma)$  on two letters. 
Either the shift has a metric discrete spectrum, in which case $\mbox{Aut}(X,\sigma)$  
consists of the powers of the shift, or the shift has a partly continuous
spectrum, in which case the ``letter-exchanging" automorphism is 
also present. In both cases all endomorphisms are automorphisms.
For constant-length substitutions on larger alphabets, those whose shifts 
have a discrete spectrum are 
precisely those that have {\em coincidences}, \cite[Thm 7]{dekking}
and all other constant-length shifts, 
including the family of  {\em bijective substitutions}, have a partly 
continuous spectrum. (See Section \ref{subsec:background} for the 
definitions of these two terms.)
 
Primitive substitution shifts are uniquely ergodic \cite{michel}, and in the 
measurable setting Lemanczyk and Mentzen \cite{lemanczyk_mentzen} show 
that any isomorphism of a bijective substitution is the composition of 
a shift with a letter-to-letter code.  Host and Parreau  \cite{host_parreau} 
extend the results of \cite{lemanczyk_mentzen},   showing that for the   
constant-length substitution shifts which do not have a purely 
discrete spectrum,  maps in the commutant\footnote{The {\em commutant} 
is the measurable analogue of the endomorphism semigroup.} are essentially 
topological and, up to a shift, have small radius.

Substitution shifts arising from primitive substitutions are a 
sub-family of shifts with sublinear complexity. We use a non-trivial 
result of Cassaigne \cite{cas} to prove Theorem \ref{thm:main_result}, 
which tells us that for infinite minimal one-sided shifts $(X,\sigma)$ 
with sublinear complexity,  automorphisms have a bounded order, and, 
modulo powers of the shift, the same is true for infinite minimal 
two-sided shifts with sublinear complexity.
There are  more general results, with similar proofs,  in recent articles, by 
Donoso, Durand, Maass and Petite in 
\cite[Theorem~1.1]{durand_petite_maass}, 
and Cyr and Kra in \cite[Theorem~1.4]{cyr_kra_2}, both for the two-sided case. 
In both of these  articles, the condition of sublinear complexity is relaxed
($\liminf p_X(n)/n$, rather than
$\limsup p_X(n)/n$, is  required to be finite).
Moreover, Cyr and Kra generalise this theorem to  non-minimal shifts. 
We prove a related one-sided result. Also, 
under the assumption of sublinear complexity, our bound gives sharper 
bounds in some cases, but requires finer
information about the branch points or asymptotic 
orbits. We give a case where the bound is tighter in Example \ref{ex:sharper}.
We also mention Olli's work  \cite{olli}, and
also Salo and T\"{o}rm\"{a}'s  \cite{salo_torma}.

In Section \ref{sec:examples_and_bounds}, we use results in the literature 
to give bounds on the orders of elements in 
$\mbox{Aut}(X,\sigma)/\langle  \sigma \rangle$ for known families 
of shifts with sublinear complexity.  We also formulate a notion of 
one-sided  coalescence in Theorem \ref{thm:one_sided_coalescence}:
Example \ref{example_three} 
tell us that not every endomorphism of a one-sided shift is an automorphism 
composed with a shift.
 
The proof of Theorem \ref{thm:main_result} is not constructive. Indeed,  
for an arbitrary shift with sublinear complexity,  
it does not seem to be  meaningful 
to ask for  an algorithm that computes its automorphism group: it is not 
clear what sort  of data it would take as input.  Nevertheless the 
question of whether there exists an algorithm to compute 
$\mbox{Aut}(X,\sigma)$ makes sense for  substitution shifts. 
In Section \ref{sec:coincidences}, we focus on this problem  for the 
family of  primitive constant  length $r$ substitutions. There we assume that our 
substitution $\theta$ is primitive, constant-length,
and generates an infinite shift. 

We reprove Theorem \ref{thm:main_result} for these shifts, 
obtaining finer information. 
Our proof techniques are a generalisation of Coven's strategy in 
\cite{coven_automorphism} for constant-length substitutions
on a two-letter alphabet. There, 
the fact that these shifts  are a somewhere one-to-one extension of their 
maximal equicontinuous factor $\Z_r$ was used to associate to 
each automorphism a ``fingerprint"  in $\Z_r$  which must satisfy 
certain constraints: we re-state this technique in  Theorem \ref{thm:ethan}. 

We show that for the case of larger alphabets a similar strategy works.
Having transferred our search to 
$\Z_r$, we prove in Lemma \ref{lem:soficsum} that the constraints 
mentioned above imply that the
fingerprints of  elements of $\mbox{Aut}(X_\theta,\sigma)$ are rational numbers
with small denominators.
We use the arithmetic properties of odometers to show in 
Corollary \ref{cor:cyclic} and the note that follows it, that  
the automorphism groups 
of a large family of  Toeplitz shifts, including those
generated by constant-length substitutions with a coincidence, with 
{\em height} one are cyclic. The same arithmetic properties imply that 
these shifts have trivial one-sided automorphism groups
(Theorem \ref{thm:trivial_one_sided_group}).
A technical detail is that not every rational that satisfies the desired 
constraints is the fingerprint of some automorphism in $\mbox{Aut}(X_\theta,\sigma)$. 
In Propositions \ref{prop:radius_two} and \ref{prop:last_nail_lemma}, we find 
checkable conditions for the existence of an automorphism.
 Our arguments 
culminate in Theorem \ref{thm:last_nail}, where we describe 
an algorithm to compute
$\mbox{Aut}(X_\theta,\sigma)$. We illustrate our constructive 
methods with some examples, 
including one where $ \langle \sigma \rangle \subsetneq \mbox{Aut}(X_\theta,\sigma)$.

With some minor modifications, we can also 
describe  in Section \ref{subsec:conjugacy_computation} an algorithm 
to compute the set of conjugacies between 
two shifts generated by constant-length substitutions. 
We note that Durand and Leroy 
\cite{DL} have communicated to us that they have  a 
result for arbitrary primitive substitutions. 
Coven, Dekking and Keane \cite{CDK}
have an alternative proof of the decidability
of topological conjugacy of constant-length substitution shifts. When we 
started the project, their decision procedure for testing
conjugacy was shown to 
terminate for many, but not all constant-length substitutions. They have since 
proved that their procedure always terminates.

We remark that there is a common thread between the arguments in Section
\ref{sec:sublin} and Section \ref{sec:coincidences}, namely that in both
cases we study an equivalence relation on the points of $X$ (asymptotic
equivalence and having a common image in the maximal equicontinuous
factor respectively). While most equivalence classes are small (of size 1
and $c$, the {\em column number}, respectively), the larger equivalence classes
have to be mapped to themselves under automorphisms, leading to 
restrictions on the collection of possible automorphisms.

\newpage 

\section{Automorphisms of minimal shifts with sublinear complexity}
\label{sec:sublin}
 
\subsection{Notation} \label{subsec:notation}
Let $\mathcal A$ be a finite alphabet, with the discrete topology, 
and let $\N_0 = \{0,1,2,\ldots \}$.
We endow $\mathcal A^{\N_0}$ and  $\mathcal A^\Z$ 
with the product topology, and let 
$\sigma:\mathcal A^{\N_0} \rightarrow \mathcal A^{\N_0}$ 
(or  $\sigma:\mathcal A^\Z \rightarrow \mathcal A^\Z$) 
denote the shift map.
We consider only infinite  minimal shifts $(X,\sigma)$, which can be either 
one- or two-sided.
An {\em endomorphism} of $(X, \sigma) $ is a map $\Phi:X\rightarrow X$ 
which is continuous, onto, and commutes with $\sigma$; if in 
addition  $\Phi$ is one-to-one, then $\Phi$ is called an 
{\em automorphism}.  We say $\Phi$ has {\em finite order m} 
if $m$ is the least positive integer  such that
$\Phi^m$ is the identity, and that $\Phi$ is a 
{\em $k$-th root of a power of the  shift} if there exists an 
$n$ such that $\Phi^k=\sigma^n$. Let ${\mbox{Aut}}(X,\sigma)$ 
denote the automorphism group of $( X,\sigma)$. If $X$ is two-sided,  
then $\mbox{Aut}(X,\sigma)$ contains the (normal) subgroup 
generated by the shift, denoted by $\langle\sigma\rangle$.
 
The {\em language} of a minimal shift $(X,\sigma)$, denoted  
$\mathcal L_X$, is the set of all finite words that occur in points of $X$. 
We denote by $p(n)$, $n \ge 1$, the {\em complexity function\/} of
$(X,\sigma)$: $p(n)$ is the number of words (also called blocks)  
in $\mathcal L_X$ of length $n$.
A symbolic system $(X,\sigma)$ has {\em sublinear complexity}
if its complexity function is bounded by a linear function.

\subsection{ Automorphisms of minimal shifts with sublinear complexity  } 
If $(\bar X, \bar\sigma)$ is a one-sided minimal shift, let $(X, \sigma)$ 
denote the two-sided version of $(\bar X, \bar\sigma)$ i.e. $X$ consists 
of all bi-infinite sequences whose finite subwords belong to  
$\mathcal L_{\bar X}$. 
Given a two-sided shift $(X, \sigma)$ we define similarly its one-sided version.
If $(\bar X,\bar\sigma)$  is minimal, then the one-sided version of 
$( X, \sigma)$  is  $(\bar X,\bar\sigma)$  itself.
Henceforth $(\bar X, \bar\sigma)$ refers to a one-sided shift and 
$(X, \sigma)$ refers to a two-sided shift.  

For $k>1$, we say that $\bar x\in \bar X$ is a {\em branch point of order $k$ }  
if $|\sigma^{-1}(\bar x)|=k$. 
If $(X,T)$ is a minimal invertible dynamical system, we define
an equivalence relation $\sim$ on orbits in $X$ as follows: 
Two orbits $\mathcal O_{x}=\{T^n(x): n\in \Z\}$ 
and $\mathcal O_{y}=\{T^n(y): n\in \Z\}$ in $X$ are 
{\em right asymptotic}, denoted $\mathcal O_{x}\sim \mathcal O_{y}$,  
if there exists $m\in\Z$ such that 
$d(T^{m+n}x,T^ny)\to 0$ as $n\to\infty$. The right asymptotic equivalence 
class of $\mathcal O_x$ will be denoted by $[x]$. A right asymptotic
equivalence class will be said to be non-trivial if it does not consist
of a single orbit. 
Clearly in the case that
$X$ is a shift,  $\mathcal O_x$ and $\mathcal O_y$ are right
asymptotic if there exists $m$ such that $x_{m+n}=y_n$ for 
all $n$ sufficiently large. 
 
\begin{lemma}\label{lem:cassaigne} Let $(\bar X, \bar\sigma)$ be an 
infinite minimal one-sided shift and let $(X,\sigma)$ be the 
corresponding two-sided shift. If $(\bar X, \bar\sigma)$ 
has sublinear complexity, then 
$(\bar X, \bar\sigma)$ has  finitely many branch points, and 
$ \{\Orb_x\colon x\in X\} $  has 
finitely many non-trivial right 
asymptotic equivalence classes. 
\end{lemma}
 
\begin{proof} 
We  prove that if $(\bar X, \bar\sigma)$ has sublinear complexity, then 
$(\bar X, \bar\sigma)$ has  finitely many branch points, 
as the latter implies that $ \{\Orb_x\colon x\in X\} $  has 
finitely many non-trivial right asymptotic equivalence classes. 

Define  for $n\geq 1$ the {\em complexity difference function} 
$s(n) := p(n+1)-p(n)$.   Since $(\bar X, \bar\sigma)$ has sublinear 
complexity,  $s(n)$ is bounded \cite{cas}, say $s(n)\leq L$ 
for all $n\geq 1$.  Then there are at most  $L$ words of length $n$ 
with at least two left extensions, and so there are at most $L$ 
branch points in $\bar X$.         
\end{proof}
 
Next we have  a basic lemma about automorphisms of minimal 
(not necessarily symbolic) dynamical systems. 

\begin{lemma}\label{lem:tool}
Let $(X,T)$ be a minimal dynamical system. 
\begin{enumerate}
\item  
Suppose $T$ is not invertible. 
If there is a point $x\in X$ and a finite subset $F$ of $X$
such that $\Phi(x)\in F$  for every automorphism~$\Phi$
of $(X,T)$, then $|\mbox{Aut}(X,T)|\le|F|$. In particular,
every automorphism has order at most $|F|$. 
\item  
Suppose  $T$ is invertible.  
If there is a point $x\in X$ and a finite collection~$\mathcal F$ 
of right asymptotic equivalence classes such that
$[\Phi(x)]\in \mathcal F$
for every automorphism $\Phi$ of $(X,T)$, then 
$|\mbox{Aut}(X,T)/\langle T\rangle|\le|\mathcal F|$. 
In particular, the order of each
element of $\mbox{Aut}(X,T)/\langle T\rangle$ is at most $|\mathcal F|$. 
\end{enumerate}
\end{lemma}

\begin{proof}
For the first part, notice that an automorphism of a minimal
dynamical system is determined by its action on a single point. 

For the second part, suppose that $\Phi$ and $\Psi$ are automorphisms 
of $(X,T)$ such that 
$\Orb_{\Phi(x)}\sim\Orb_{\Psi(x)}$ for some $x$. Let $m$ be
such that $d(T^n(\Phi(x)),T^{n+m}(\Psi(x)))\to 0$. 
Then we claim that $\Phi=\Psi\circ T^m$. For any $x'\in X$, 
let $(n_i)$ be an increasing sequence of integers so that
$T^{n_i}x\to x'$. Now we have 
\begin{align*}
\Phi(x')&=\lim_{i\to\infty} \Phi(T^{n_i}x)=\lim_{i\to\infty} T^{n_i}\Phi(x)\\
&=\lim_{i\to\infty} T^{n_i+m}\Psi(x)=\lim_{i\to\infty} \Psi(T^m(T^{n_i}x))=
\Psi(T^mx'),
\end{align*}
where for the third equality, we used that fact that $\Orb_{\Phi(x)}$
and $\Orb_{\Psi(x)}$ are right asymptotic. Hence we deduce that
an automorphism of $X$ is determined up to composition with a 
power of $T$ by the right asymptotic
equivalence class of the image of a single point. 
\end{proof}
 
The following result tells us that for  infinite minimal shifts 
with sublinear complexity, 
$|\mbox{Aut}(X,\sigma)/\langle \sigma \rangle|$ is bounded by the 
number of maximal sets of mutually right asymptotic orbits.

\begin{theorem} \label{thm:main_result} 
Let $(\bar X, \bar\sigma)$ be an infinite minimal one-sided shift  
and let $(X,\sigma)$ be the corresponding 
two-sided shift.  For $k>1$, let $\bar M_k$ be the number of $k$-branch 
points in $\bar X$, and $M_k$ be the number of  
$\sim$-equivalence classes of size $k$ in $X$. Suppose that 
$(\bar X, \bar\sigma)$ has sublinear complexity, so that
both $\bar M_k$ and $M_k$ are finite. Then
\begin{enumerate}
\item
 $\mbox{Aut}(\bar X, \bar\sigma)$ has at most $ \bar M :=\min\{ \bar M_k\colon
\bar M_k>0\}$ elements, and 
\item $\mbox{Aut}(X,\sigma)/\langle\sigma\rangle$ has at most 
$M:=\min \{M_k\colon M_k>0\}$ elements. 
\end{enumerate}
\end{theorem}
 
\begin{proof}
Using Lemma \ref{lem:cassaigne}, $\bar X$ has finitely many branch points
and $\{\Orb_x\colon x\in X\}$ has finitely many 
non-trivial right asymptotic equivalence classes. 
The fact that the systems are infinite implies 
that $\bar X$ (respectively $X$) has at least one branch point (respectively
one non-trivial right asymptotic 
equivalence class) so that $\bar M$ and $M$ are positive and finite. 

Let $\bar M=\bar M_k$. Let $\bar x$ be a branch point of order $k$. 
An automorphism $\bar \Phi$ of 
$(\bar X, \bar \sigma)$ must map branch points of order $k$ to   
branch points of order $k$, so we apply Lemma \ref{lem:tool}(1), 
to the set of branch points of order $k$, to obtain (1).
 
Similarly, let $M=M_\ell$ and let $x$ be such that $\Orb_x$ belongs
to an equivalence class consisting of $\ell$ orbits. If $\Phi$ is an
automorphism $(X,\sigma)$, then $\Orb_{\Phi(x)}$ must also belong to
an equivalence class consisting of $\ell$ orbits. Hence by
Lemma \ref{lem:tool}(2), there are at most $M_\ell$ automorphisms
of $(X,\sigma)$ up to composition with a power of
$\sigma$.
\end{proof}
 
As $\mbox{Aut} ( \bar X, \bar \sigma)$ and  
$\mbox{Aut} ( X,\sigma)/\langle\sigma\rangle$ are both finite groups, 
the order of any element of $\mbox{Aut} (\bar X, \bar\sigma)$ divides 
$|\mbox{Aut} (\bar X, \bar\sigma)|$, and any element of  
$\mbox{Aut} (X, \sigma)$ is a $k$-th root of the shift, 
where $k$ divides $|\mbox{Aut} (X, \sigma)/\langle\sigma\rangle|$.
 
\subsection{Examples and bounds}\label{sec:examples_and_bounds}

\subsubsection{Substitution shifts.}\label{subsec:subs}
A  {\em substitution} is a map from $\mathcal A$ to the set of 
nonempty finite words on $\mathcal A$. We use concatenation to 
extend $\theta$ to a map on finite and infinite words 
from $\mathcal A$. We say that $\theta$ is {\em primitive} if there is some 
$k\in \N$ such that for any $a,a'\in \mathcal A$,
the word $\theta^k(a)$ contains at least one  occurrence of $a'$. 
By iterating $\theta$ on any fixed letter in $\mathcal A$, 
we obtain  one-sided (right) infinite points $u=u_0 \ldots $  
such that $\theta^j(u)=u$ 
for some natural $j$. 
The pigeonhole principle implies that  $\theta$-periodic points
always exist, and, for primitive substitutions, we define 
$\bar X_\theta$ to be the shift orbit closure of any one of 
these $\theta$-periodic points and call 
$(\bar X_\theta, \bar \sigma)$ a {\em one-sided substitution shift}.
Barge, Diamond and Holton  \cite{bdh} show that  a primitive 
substitution on $k$ letters has  at most $k^2$ right 
asymptotic orbits, so that we can apply Theorem \ref{thm:main_result} 
to deduce that any automorphism of $( X_\theta, \sigma)$ is a $j$-th root 
of a power the shift for some $j\leq k^2$.
 
\begin{example}\label{ex:sharper}
Let $\mathcal A = \{a_n, b_n, c_n: 1 \leq n \leq N  \}$.  
Let $W$ be a word which contains all of the letters 
in $\mathcal A$.  Define a substitution,  $\theta$, by 
$\theta(a_n) = Wa_na_n$, $\theta(b_n) = Wb_na_n$ and 
$\theta(c_n) = Wc_na_n$ for
$1\leq n \leq N$. Now the one-sided 
shift $(\bar X_\theta,\bar\sigma)$ contains $N$ branch points of 
order three, $\{  x^n: 1\leq n \leq N  \}$, where $x^n$ satisfies 
the equation $a_n \theta( x^n) = x^n$. 
Also, $\bar X_\theta$ contains only
one other branch point of order $N$: the unique right-infinite $\theta$-fixed 
point.    Similarly, in the two-sided 
shift $(X_\theta,\sigma)$, there are exactly $N$ distinct right 
asymptotic equivalence classes of size 3, and one  right 
asymptotic equivalence class of size $N$.
Now our bounds from Theorem \ref{thm:main_result} tell us 
that each of $\mbox{Aut}(\bar X_\theta, \bar \sigma)$ and 
$\mbox{Aut}( X_\theta, \sigma)/\langle \sigma \rangle$ consist of one element.

We compare our bounds to those that would be obtained using 
the results in \cite{cyr_kra_2} or  \cite{durand_petite_maass}. In the 
latter, Theorem 3.1 tells us that  $|\mbox{Aut}(X) / \langle\sigma\rangle|$ 
divides $3N+1$, the number of non-trivial right asymptotic 
equivalence classes. In the former, Theorem 1.4 tells us that 
if $\liminf p(n)/n<k$, then $|\mbox{Aut}(X) / \langle\sigma\rangle|<k$. 
If one takes the word $W=a_1b_1c_1 a_2b_2c_2  \dots a_Nb_Nc_N$, 
a crude check shows that $\liminf p(n)/n\geq 3N/2$, so that the 
upper bound for $|\mbox{Aut}(X) / \langle\sigma\rangle|$ coming from 
Theorem 1.4 of \cite{cyr_kra_2} would be at least $3N/2$.

\end{example}

\subsubsection{Linearly recurrent shifts}
\label{subsec:linrec}
The notions below are the same for one-sided and two-sided shifts, 
so we will state them only for two-sided shifts, thus avoiding the bars.
If $u,w \in \mathcal L_X$, we that $w$ is a {\em return word to $u$} 
if (a) $u$ is a prefix of ~$w$,  (b) $wu \in \mathcal L_X$, and (c) 
there are exactly two occurrences of $u$ in  $wu$. Letting  $\ell(u)$ 
denote the length of $u$, a minimal shift $(X,\sigma)$ is 
{\em linearly recurrent} if there exists a  constant $K$ such 
that for any word $u \in \mathcal L_X$
and any return word (to $u$) $w$,  $\ell(w)\leq K\ell(u)$.  
Such a ~$K$ is called a {\em recurrence constant.} 
Durand, Host and Skau  show that  
if $(X,\sigma)$ has linear recurrence 
constant $K$, then its complexity is bounded above by $Kn$ for large 
$n$ \cite[Theorem~23]{dhs}. 
Cassaigne \cite{cas} shows that if 
$p(n)\leq Kn+1$, then  there are at most $2K(2K+1)^2$ branch points 
and asymptotic orbits, and we can now apply Theorem 
\ref{thm:main_result} to bound $|\mbox{Aut}(\bar X, \bar \sigma)|$ and 
$|\mbox{Aut}( X, \sigma)/\langle\sigma\rangle|$ above by   $2(K+1)(2K+3)^2$.   

Durand \cite[Corollary~18]{fdlr} shows that for linearly recurrent 
two-sided shifts, any endomorphism is an 
automorphism.  The corresponding one-sided 
version of this is:
\begin{theorem}\label{thm:one_sided_coalescence}
Let $(\bar X,\bar \sigma)$ be an infinite, minimal,   
linearly recurrent one-sided 
shift  with recurrence constant $K$.  Then every endomorphism of 
$(\bar X, \bar \sigma )$ is a $k$-th root of a power of the shift 
for some  positive 
$ k\leq   2(K+1)(2K+3)^2$.
\end{theorem}
\begin{proof}
Any endomorphism $\bar\Phi$ of  $(\bar X,\bar\sigma)$ defines an endomorphism 
$ \Phi$ of $( X, \sigma)$, which must be an automorphism by 
\cite[Corollary~18]{fdlr}. By Cassaigne's result, 
$\Phi^k =\sigma^n$ for some positive $k \leq    2(K+1)(2K+3)^2     $
and some $n\in \Z$, so $ \bar\Phi^k =\bar\sigma^n$.
\end{proof}

In Example \ref{example_three} we describe a non-trivial 
endomorphism of $(\bar X, \bar \sigma)$ 
which is not a power of the shift. 

\section{Constant length substitutions}\label{sec:coincidences}

Let $\theta$ be a substitution on the alphabet $\mathcal A$. 
The substitution $\theta$ has
{\em (constant) length~r}  if for each $a\in \mathcal A$, 
$\theta (a)$ is a word of length $r$.
In this section we generalise the results of Coven in \cite{coven_automorphism} 
to primitive constant-length substitutions on a finite alphabet. 
We then provide algorithms to compute the automorphism 
group of a constant-length substitution shift, and also the set of 
conjugacies between two constant-length substitution shifts.

\subsection{The $r$-adic integers}
Let $\Z_r$ denote the set of $r$-adic integers, identified with the
one-sided shift space consisting of sequences $(x_n)_{n\ge 0}$ with
the $x_n$'s taking values in $\{0,1,\ldots,r-1\}$. We think of 
these expansions as being
written from right to left, so that $\Z_r$ consists of left-infinite sequences
of digits, $x_0$ is the rightmost digit of $x$, and
so that addition in $\Z_r$ has the carries 
propagating to the left in the usual way.
Formally $\Z_r$ is the inverse limit of rings 
$\Z/r^n\Z$, and is itself a ring.

We make no assumption
about the primality of $r$. 
If $r$ is not a prime power, then $\Z_r$ has zero divisors.
Nevertheless, if $k\in \Z$   and 
 $\gcd(k,r)=1$, then multiplication by $k$
is an isomorphism of $\Z_r$ as an additive group, so that one can make 
sense of expressions such as $m/k$ for $m\in\Z_r$ if $\gcd(k,r)=1$. 
If $k\in \Z$ and $\gcd(k,r)>1$, then multiplication by $k$ is still 
an injection of $\Z_r$.
A version of the standard argument shows that if $m$ and $k$ are integers
with $\gcd(k,r)=1$, then 
$m/k$ has an eventually periodic digit sequence; and conversely 
any point of $\Z_r$
with an eventually periodic digit sequence can be written 
in the form $m/k$ for $m$
and $k$ in $\Z$ with $\gcd(k,r)=1$. We may naturally think of $\Z$ 
as a subset of $\Z_r$: a point  $s$ in $\Z_r$ 
whose digits are eventually  
all 0's or eventually all $(r-1)$'s is the $r$-adic expansion of an integer:
we write $s\in \Z$. Note also that as an additive group, $\Z_r$ is torsion free.

We need the following fact.

\begin{lemma}\label{lem:divr}
Let $kt\in\Z$ for some $t\in\Z_r$ and $k\in\Z\setminus\{0\}$.
Then there exists $m|k$ such that $mt\in\Z$ and $\gcd(m,r)=1$. 
\end{lemma}

\begin{proof}
Let $k=lm$, where $\gcd(m,r)=1$ and $l|r^n$. Let $p:=r^n/l$ and $s:=mt$. 
Now $r^ns=pkt\in\Z$.
However, $r^n s \in r^n\Z_r$, which
is the set of elements of $\Z_r$ with at least $n$ trailing 0's.
Dividing by $r^n$ removes these 0's, showing
that $s\in\Z$ as required.
\end{proof}

\begin{lemma}\label{lem:cyclic}
Let $H$ be a subgroup of $\Z_r$ containing $\Z$ generated by 
a collection of elements $(p_i/q_i)_{i=1}^n$ with $\gcd(p_i,q_i)=\gcd(q_i,r)=1$.
Then $H$ is generated by a single element of the form $1/q$.
\end{lemma}

\begin{proof}
Let $\alpha_i=p_i/q_i$ and $q=\mbox{Icm}(q_1,\ldots,q_n)$, 
so that if $\tilde p_i:=p_i(q/q_i)$ 
we have $q\alpha_i=\tilde p_i$ for $i=1,\ldots,n$.
Note that if $m$ is a prime and $m^\alpha|q$, with $\alpha$ maximal,  
then $m^\alpha|q_i$ for some $i$, 
so that $m\nmid p_i$ and $m\nmid \tilde p_i$. 
Hence $\gcd(q,\tilde p_1,\ldots,\tilde p_n)=1$ and there 
exist integers $k_0,k_1,\ldots,
b_n$ such that $k_0q+k_1\tilde p_1+\ldots k_n\tilde p_n=1$. 
Then $k_0+k_1\alpha_1+\ldots +k_n\alpha_n=1/q\in H$ and
all of the generators are multiples of $1/q$, so that $H$ is generated by
$1/q$. 
\end{proof}

\subsection{Maximal equicontinuous factors}

If $(X,T)$ is a continuous dynamical system, we say that the dynamical system 
$(Y,S)$ is the {\em maximal equicontinuous factor  of $(X,T)$} if 
$(Y,S)$ is an equicontinuous factor of $(X,T)$ with the property 
that any other equicontinuous factor $(Z,R)$ of $(X,T)$ is a factor 
of $(Y,S)$. 
The maximal equicontinuous factor of a minimal transformation is a 
rotation on a compact monothetic topological group, that is a group $G$
for which there exists an element $a$ such that the subgroup generated
by $a$ is dense. Such a group is always abelian 
\cite{downarowicz} and we will write the group operation additively.

The systems we consider will have maximal equicontinuous factor
$(\Z_r,+1)$  or 
$(\Z_r\times \{0,\ldots,h-1\},+1)$, where the group operation 
on $\Z_r\times \{0,\ldots,h-1\}$ is 
$$
(x,i)+(y,j)=\begin{cases}
(x+y,i+j)&\text{if $i+j<h$;}\\
(x+y+1,i+j-h)&\text{otherwise,}
\end{cases}
$$
and where $r$ is not necessarily prime.

Let $\mbox{End}(X,\sigma)$ be the set of endomorphisms of $(X,\sigma)$.   
Given two shifts $(X,\sigma)$  and $(Y,\sigma)$, let $\mbox{Conj}(X,Y)$ 
denote the (possibly empty) set of topological conjugacies between  
$(X,\sigma)$  and $(Y,\sigma)$, and let $\mbox{Fac}(X,Y)$ denote the set
of factor maps from $(X,\sigma)$  to $(Y,\sigma)$.
For completeness we prove the following, which is a straightforward 
generalization of work done by Coven in  \cite[\S3]{coven_automorphism}.
 
\begin{theorem}\label{thm:ethan}  
Let  $(X,\sigma)$  and $(Y,\sigma)$ be  infinite minimal 
shifts.
Suppose that the group rotation $(G,R)$ is the maximal 
equicontinuous factor of both $(X,\sigma)$ and $(Y,\sigma)$ and let
$\pi_X$ and $\pi_Y$ be the respective factor maps.
Then there is a map $\kappa:  \mbox{Fac}(X,Y) \rightarrow G$
such that \[\pi_Y(\Phi(x))=\kappa(\Phi)+\pi_X(x) \] for all 
$x\in X$ and $\Phi  \in \mbox{Fac}(X,Y)$. Also 
\begin{enumerate}
\item
 if $(Z,\sigma)$ is another 
shift which satisfies the assumptions on $(X,\sigma)$, then
$\kappa(\Psi\circ\Phi)=\kappa(\Psi)+\kappa(\Phi)$ for 
$\Phi \in \mbox{Fac} (X,Y)$,  $\Psi \in
\mbox{Fac}(Y,Z)$, and 
\item if $\min_{g\in G}|\pi_X^{-1}(g)|=\min_{g\in G}|\pi_Y^{-1}(g)|=c<\infty$,
then
\begin{enumerate}
\item for each $\Phi\in \mbox{Fac}(X,Y)$, we have  
\[ \{ z \in \Z_{r}: |\pi_Y^{-1}(z)|>c \}\subset
\{z \in \Z_{r}: |\pi_X^{-1}(z)|>c \} + \kappa(\Phi),
\]

and 
\item 
$\kappa$ is at most  $c$-to-one. In particular, if $\pi_X$ and 
$\pi_Y$ are somewhere one-to-one, then $\kappa $ is an injection.
\end{enumerate} 
\end{enumerate}
\end{theorem}

\begin{proof}
First we show that given any  factor map $\Phi:X\rightarrow Y$, 
there exists $\kappa(\Phi) \in \Z_r$ such that $\pi_X(x) + 
\kappa(\Phi)= \pi_Y(\Phi(x))$. Fix  any $b\in X$, and define 
$\kappa (\Phi):= \pi_Y(\Phi(b))- \pi_X(b)$. 
Notice that $f(x)=\pi_Y(\Phi(x))-\pi_X(x)$ is a 
$T$-invariant continuous function. Since $T\colon X\to X$
is minimal, it follows that $f(x)=f(b)=\kappa(\Phi)$ for all $x\in X$.

 Thus for each factor mapping $\Phi:(X,\sigma)\rightarrow (Y,\sigma)$ 
there is a map $F:G\to G$, namely $F(g) := 
g+\kappa(\Phi)$, such that 
$\pi_Y\circ \Phi= F\circ \pi_X$. This gives us a map $\kappa: 
\mbox{Fac}(X,Y) \rightarrow G$.

With the assumptions of  (1), note that  $\pi_Z (\Psi\circ \Phi(x))= 
\kappa(\Psi\circ \Phi) + \pi_X(x)$, but also
$\pi_Z (\Psi\circ \Phi(x)) = \kappa(\Psi)+\pi_Y(\Phi(x)) =  
\kappa(\Psi) +\kappa(\Phi) + \pi_X(x)$, and (1) follows.

Suppose that $\Phi\in\mbox{Fac}(X,Y)$. Let $|\pi_Y^{-1}(z)|>c$,
so that there exist distinct $y_1,\ldots,y_{c+1}$ in $Y$ with $\pi_Y(y_i)=z$
for each $i$. Since $\Phi$ is surjective, there exist $x_1,\ldots,x_{c+1}$
with $\Phi(x_i)=y_i$, and we have $\pi_X(x_i)=z-\kappa(\Phi)$ for each $x_i$. 
Hence $|\pi_X^{-1}(z-\kappa(\Phi))|>c$, so 
$z\in \{z'\in\Z_r\colon |\pi_X^{-1}(z')|>c\}+\kappa(\Phi)$.
 
Let $t$ be in the range of $\kappa$, let
$z\in \Z_r$ satisfy $|\pi_Y^{-1}(z)|=c$ and let
$x\in\pi_X^{-1}(z-t)$. Then if $\kappa(\Phi)=t$,
we have $\Phi(x)\in \pi_Y^{-1}(z)$. Since a factor map of a minimal
system is determined by its action on a single point, we deduce
that
there are at most $c$  factor maps in $\mbox{Fac}(X,Y)$, proving (2b).
\end{proof}

\subsection{Background results on constant-length substitutions}
\label{subsec:background}
In this section, we collect some results that we use later to reduce
the case of general constant-length substitutions to cases that can
be more straightforwardly handled.

Let $\theta$ be a primitive length $r$ substitution with fixed 
point $u$, and with $(X_\theta,\sigma)$  infinite. 
The {\em height }$h=h(\theta)$ of $\theta$   is defined as 
\[
h(\theta):= \max \{n\geq 1: \gcd(n,r)=1, n | \gcd\{a: u_a=u_0 \} \}\, .
\]
If $h>1$, this means that $\mathcal A$ decomposes into $h$ disjoint 
subsets: $\mathcal A_1\cup\ldots\cup \mathcal A_h$, where a symbol from
$\mathcal A_i$ is always followed by a symbol from $\mathcal A_{i+1}$.
Such a system is a constant height suspension. 
In fact, as shown in \cite[Remark 9, Lemmas 17 and 19]{dekking},
such a system is conjugate to
a constant height suspension of another constant-length substitution.

\begin{proposition}\label{prop:height}
Let $\theta$  be a primitive,  length $r$ substitution defined on 
$\mathcal A$, such that $(X_{\theta},\sigma)$ is infinite. Then
 \begin{enumerate}
\item
$ h < |\mathcal A|$, 
\item
there is a  primitive length $r$ substitution shift $(X_{\theta'},\sigma')$ of height one such that 
$(X_{\theta}, \sigma) 
\cong(X_{\theta'}\times \{0,\ldots,h-1\},T)$ where

\begin{equation*}
 T(x,i):=
\begin{cases}
  (x,i+1)      & \text{ if  } 
0\leq  i<h-1    \\
    (\sigma'(x), 0)   & \text{ if } i=h-1
\end{cases}
\end{equation*} 
\end{enumerate}
Furthermore  $h$, $\theta'$, and the conjugacy between  $(X_{\theta}, \sigma)$ 
and $(X_{\theta'}\times \{0,\ldots,h-1\},T)$
 can be determined algorithmically.
\end{proposition}

The substitution $\theta'$ described in Proposition \ref{prop:height} is called a 
{\em pure base} of $\theta$.
 
For every  $\Psi \in \mbox{Aut}(X_{\theta'},\sigma')$ and  $0\leq i\leq h-1$
we can define $\Psi_i \in \mbox{Aut}(X_{\theta'} \times \{0,\ldots,h-1\}, T)$ as

\begin{equation*}
\Psi_i(x,j) :=
\begin{cases}
  (\Psi(x), j+i) & \text{ if  } 
   j+i<h \\
       (\Psi(\sigma'(x)), j+i \mod h)     & \text{ if } j+i\geq h.
\end{cases}
\end{equation*}

\begin{proposition}\label{prop:automorphism_with_height}
Let $\theta$  be a primitive,  length $r$ substitution, of height $h$, and  
such that $(X_{\theta},\sigma)$ is infinite. Let $\theta'$ 
be a pure base of $\theta$.
Then $ \mbox{Aut}(X_{\theta},\sigma) =
\{   \Psi_j: \Psi \in \mbox{Aut}(X_{\theta'},\sigma') \mbox{ and }  0\leq j<h  \}.$ 
\end{proposition}

\begin{proof}
By Proposition \ref{prop:height}, $(X_{\theta},\sigma)$ is conjugate to  
$(X_{\theta'} \times \{0,\ldots,h-1\}, T)$, and so  
$\mbox{Aut}(X_{\theta},\sigma)$ 
is group isomorphic to $\mbox{Aut}(X_{\theta'} \times \{0,\ldots,h-1\}, T)$. 
By the definition of $\Psi_i$, we see that 
$\{   \Psi_j: \Psi \in \mbox{Aut}(X_{\theta'},\sigma') \mbox{ and }  
0\leq j<h  \}\subset \mbox{Aut}(X_{\theta'} \times \{0,\ldots,h-1\},T)$.

Conversely, given $\Psi\in \mbox{Aut}(X_{\theta'} \times 
\{0,\ldots,h-1\}, T)$, 
suppose that  $\Psi(x,0) \in X_{\theta'}\times \{i\}$
for some $x\in X_{\theta'}$. Define $f(x)=\pi_2(\Psi(x,0))$, where $\pi_2$
is the projection onto the second coordinate and notice that $f$
is a continuous $\sigma'$-invariant function on $X_{\theta'}$ and hence is constant.
This, and the fact $\Psi$ commutes with $T$ implies
that $\Psi = \Psi'_i$ for some $\Psi'\in\mbox{Aut}(X_{\theta'},\sigma')$. 
\end{proof}

Proposition \ref{prop:automorphism_with_height} tells us that in order 
to compute the automorphism group of a primitive constant-length 
substitution, or, as we shall see later,  the set of conjugacies 
between two such substitutions, it is sufficient to work with  
with a pure base $\theta'$ of $\theta$. Thus, apart from  the 
statements of our main results Theorems \ref{thm:last_nail} and 
\ref{thm:last_nail_conjugacy_theorem}, and Corollary \ref{cor:cyclic}, we henceforth assume that 
our substitutions are of height one.

\vspace{1cm}
Let $\theta$ be a length $r$ substitution. We write 
$\theta (a)= \theta_0(a) \ldots \theta_{r-1}(a)$; 
with this notation we 
see that for each  $0\leq i \leq r-1$, we have a map
$\theta_i:\mathcal A \rightarrow \mathcal A$ where $\theta_i(a)$ is the 
$(i+1)$-st letter of $\theta(a)$.

Let $\theta$ have pure base $\theta'$.
We say that $\theta$ {\em has column number $c$ } if for some $k\in \N$, 
and some $(i_1, \ldots ,i_k)$,  
$|\theta'_{i_1}\circ \ldots \circ\theta'_{i_k}(\mathcal A)|=c$ and $c$ 
is the least such number. 
In particular, if $\theta$ has column number one, then we will say 
that  $\theta$ {\em has a coincidence}. If for each $0\leq i \leq r-1$ 
and $a\neq b$,  $\theta_i(a) \neq \theta_i(b)$, we say $\theta$ is 
{\em bijective}. 
Henceforth we shall be working with primitive substitutions $\theta$  
which are length $r$  and such that $(X_\theta,\sigma)$ is infinite.

The following is shown by Dekking \cite{dekking}, with partial 
results by Kamae \cite{kamae} and Martin \cite{martin}.
\begin{theorem} \label{thm:dekking}
Let $\theta$  be a primitive,  length $r$ substitution, of height one, 
and such that $(X_\theta,\sigma)$ is infinite. Then 
the maximal equicontinuous factor of $(X_\theta, \sigma)$ is $(\Z_r , +1)$.
\end{theorem}

Recall the map $\kappa$ as defined in Theorem \ref{thm:ethan}.

\begin{corollary}\label{cor:cyclic}
Let $\theta$  be a primitive,
length $r$ substitution, 
such that $(X_{\theta},\sigma)$ is infinite.
\begin{enumerate}
\item Let  $\theta$ have height one.  If $\kappa$ is injective (in particular if $\theta$ has a coincidence), then
$\mbox{Aut}(X_{\theta} ,\sigma) $ is cyclic.
\item Let $\theta$ have height $h>1$ and $\theta'$ be the pure 
base of $\theta$. Suppose that $\kappa: X_{\theta'}\rightarrow\Z_r$ is injective,
so that $\mbox{Aut}(X_{\theta'} ,\sigma')=\langle \Phi \rangle $ with 
$\kappa(\Phi)=1/k$. Then
 $\mbox{Aut}(X_\theta,\sigma)$ is abelian and generated by
$T$ and $\Phi_0$ (in the above notation).
Further,  
$\mbox{Aut}(X_{\theta},\sigma)$ is cyclic if and only if $\gcd(k,h)=1$.
\end{enumerate}
\end{corollary}

\begin{proof}
Let $\theta$ be as in the statement, and first  suppose that $\theta$ has height one, so that its maximal 
equicontinuous factor is $(\Z_r,+1)$. By the discussion in 
Section \ref{subsec:subs} and Theorem \ref{thm:main_result},
$\mbox{Aut}(X_\theta,\sigma)$ is finitely
generated. Let a set of generators be $\Phi_0=\sigma,\Phi_1,\ldots,
\Phi_n$, say. Then Theorem \ref{thm:main_result} implies 
$\kappa(\Phi_0)=1$ and $\kappa(\Phi_i)=p_i/q_i$ for each $i$
for some $p_i\in\Z$ and $q_i\ge 1$. By Lemma \ref{lem:cyclic},
$\kappa(\mbox{Aut}(X_\theta,\sigma))$ is cyclic. 
The result follows from the injectivity of $\kappa$. In particular, 
since substitutions with coincidences are precisely those whose 
shifts are a somewhere one-to-one extension of  their maximal 
equicontinuous factor \cite{dekking},
we deduce by Theorem \ref{thm:main_result}
that if $\theta$ has a coincidence, then $\mbox{Aut}(X_\theta,\sigma)$ is cyclic.

If $\theta$ has height $h>1$, let $\theta'$ be the pure base of $\theta$. 
We assume that $\theta'$ has a coincidence. By the above, $\mbox{Aut}(X_{\theta'},\sigma')$
is cyclic, and generated by some $\Phi$ with $\kappa(\Phi)=1/k$ for some
$k\in\N$. Suppose that $\gcd(h,k)=1$ and let $ah+bk=1$ with $0\le b<h$. 
Then consider the automorphism $\Psi$ of 
$X_{\theta'}\times\{0,1,\ldots,h-1\}$ given 
by $\Psi(x,i)=(\Phi^a(x),i+b)$ if $i+b<h$ or 
$(\Phi^a(\sigma'(x)),i+b-h)$ otherwise (that is $\Psi=(\Phi^a)_b$ in the
previous notation). Define an equivalence relation on 
$X_{\theta'}\times\Z$ to be the transitive closure of $(\sigma'(x),i)\sim(x,i+h)$
so that $X_{\theta'}\times\{0,\ldots,h-1\}$ is a system of representatives of the 
equivalence classes. In this notation, $\Psi(x,i)=(\Phi^a(x),i+b)$. 
Notice that $\Psi^k(x,i)=(\Phi^{ak}(x),i+kb)=(\sigma'^ax,i+kb)
\sim(x,i+kb+ah)=(x,i+1)=T(x,i)$,
so that $T\in\langle
\Psi\rangle$. Since $\gcd(a,k)=1$, some power of $\Phi^a$ is of the
form $\Phi\circ {\sigma'}^m$. Hence $\Phi_0\in \langle\Psi\rangle$ also,
so that $\mbox{Aut}(X_\theta,\sigma)$ is cyclic as required.

Conversely if $\gcd(h,k)=d>1$, then let
$\Psi=\Phi_0^{k/d}\circ T^{-h/d}$. We have $\Psi\ne\mbox{Id}$,
but $\Psi^d=\mbox{Id}$. Since
infinite cyclic groups are torsion-free, we see that 
$\mbox{Aut}(X_\theta,\sigma)$ is not cyclic. 

\end{proof}

Substitutions for which $\gcd(h,k)>1$ can be constructed. For, 
starting with a (primitive, nonperiodic) substitution with a 
coincidence on two letters, its automorphism group equals 
$\Z$ \cite{coven_automorphism}. We can use the 
techniques in Example \ref{example_three} to obtain a substitution 
shift $(X_\theta', \sigma)$, the $\kappa$ values of whose automorphism 
group is $\langle\frac{1}{k}\rangle$. Now building a tower of 
height $h$ over $(X_\theta', \sigma)$ gives us the desired substitution shift.
We remark also  that Part (1) of Corollary \ref{cor:cyclic} holds, with 
essentially the same
proof, for linearly recurrent {\em Toeplitz} systems 
\cite{downarowicz} whose maximal equicontinuous factor is $\Z_r$.

A constant-length substitution is called \emph{injective} if $\theta(i)\ne
\theta(j)$ for any distinct $i$ and $j$ in $\mathcal A$. 
It will be convenient to deal with injective substitutions in what follows. 
The following theorem of Blanchard, Durand and Maass allows us to restrict 
our attention to that situation.

\begin{theorem}[\cite{BDM}]\label{thm:injective}
Let $\theta$ be a constant-length substitution such that $X_\theta$ is 
infinite. Then there exists an injective substitution $\theta'$
such that $(X_{\theta},\sigma)$ is topologically conjugate to 
$(X_{\theta'},\sigma)$. 
Further, the conjugacy and its inverse are explicitly computable, 
so that $\mbox{Aut}(X_{\theta} ,\sigma) $ can be determined algorithmically from
$\mbox{Aut}(X_{\theta'} ,\sigma) $.
\end{theorem}

\begin{proof}
We only sketch the proof that the inverse of the conjugacy is
computable, as this is not described in \cite{BDM}. There, we see 
that $(X_{\theta},\sigma)$ is conjugate to $(X_{\theta'},\sigma)$ 
via $\tau$, which is a composition of less than $|\mathcal A|$
algorithmically-determined letter-to-letter maps, 
$\tau= \tau_1\circ \ldots \tau_N$. 
By iteration, it suffices to show that a single one of these maps
has a computable inverse, so we assume that $N=1$ and $\tau_1=\tau$.

By definition, $\theta\circ \tau=\theta$, so that to 
explicitly describe $\tau^{-1}$, we only need explicitly 
describe $\theta^{-1}$. If $\theta$ is a constant-length substitution
on $\mathcal A$, 
then there exists  $k\leq |\mathcal A|^2$ such that if
$\theta^k(a)=\theta^k(b)$, then $\theta^{k-1}(a)=\theta^{k-1}(b)$. 
Fix any letter $a$. The bilateral  {\em recognizability} of 
$\theta^k$ \cite[Definition 1.1 and Theorem 3.1]{mosse1} implies 
that there exists an $\ell\geq k$ such that $\theta^\ell (a)$ 
does not appear starting at the interior of any $\theta^k$ word.
Such an $\ell$ can be obtained by a simple algorithm. 
Uniform recurrence of $(X_{\theta},\sigma)$ implies that if  
$L$ is large enough, then any word of length $L$ in 
$\mathcal L_{X_\theta}$ contains   some $\theta^\ell (a)$ 
as a subword. Also, this $L$ is computable \cite[Proposition 25]{dhs}. 
Thus $\theta^{-1}$ can be explicitly described and the proof
of our claim is complete.
\end{proof}

\subsection{The description of $\{z \in \Z_{r}: |\pi^{-1}(z)|>c \}$}
\label{subsec:transgraph}
We continue to assume that $\theta$ is a length $r$ substitution of height one,
so that $\Z_r$ is the maximal equicontinuous factor of $(X_\theta,\sigma)$. 
We let $c$ be the column number of $\theta$. We show that all points of 
$\Z_r$ have at least $c$ preimages and most points of $\Z_r$ have
exactly $c$ preimages. We study the structure of the 
subset of $\Z_r$ having excess preimages.

If $P_0:=\theta^n(X_\theta)$, then $P_0$ generates 
a $\sigma^{r^n}$-cyclic partition of size $r^n$ \cite[Lemma II.7]{dekking}.
We use the notation of  \cite{kamae}, \cite{dekking}, using $\Lambda_{r^n}$ to 
denote the equivalence relation whose classes are the members of this cyclic 
$\sigma^{r^n}$-partition, i.e. we  write $\Lambda_{r^n}(x)=i$ if 
$x\in \sigma^{i} (P_0)$.

Using the maps $\Lambda_n $, we can define a maximal equicontinuous 
factor map $\pi:X_\theta \rightarrow \Z_r$.
Note that if  $\Lambda_n(x) = i$, then $\Lambda_{n+1}(x) \equiv i \mod r^{n}$.
Using this we can define $\pi(x):= \ldots x_2\,x_1\,x_0$ where for each 
$n\in \N$, $\Lambda_n(x)=\sum_{i=0}^{n-1} r^{i}x_i$. 

\begin{definition}\label{def_X_F} Let $\theta$ be a length $r$ 
substitution with column number $c$ and of height one.
Let $\Sigma_\theta$ be the one-sided shift on $\{0,1, \ldots, r-1 \}$, 
whose points are left-infinite, and whose set of forbidden words is 
\[ 
\mathcal F_\theta:= \{  w=w_k \ldots w_1 \in (\Z / r\Z)^+: 
|\theta_{w_1}\circ \ldots \circ\theta_{w_k}(\mathcal A)|=c   \}.
\]
\end{definition}
Note that $\mathcal F_\theta\neq \emptyset$. It can happen 
that $\Sigma_\theta=\emptyset$: this is the case for the 
Thue-Morse substitution: $\theta(0)=01, \theta(1)=10$.  

\begin{lemma}\label{lem:sofic}
Let $\theta$  be a primitive, length $r$ substitution  
on $\mathcal A$ with column number $c$, of height one, and 
such that $(X_\theta,\sigma)$ is infinite. Then 
$\Sigma_\theta$ is either empty or a sofic shift. 
\end{lemma}

\begin{proof}
Let $S=\{\mathcal A\}\cup 
\{\theta_{w_1}\circ\cdots\circ\theta_{w_k}(\mathcal A)\colon k\ge 1;\;
w_1,\ldots,w_k\in \{0,1,\ldots,r-1\}\}$ and 
$V=\{A\in S\colon |A|>c\}$. This is the vertex set of a directed labelled graph. 
If $A\in V$ and $\theta_i(A)\in V$, then we add an edge from $\theta_i(A)$ to $A$
labelled with the symbol $i$. We claim that the set $\Sigma_\theta$ 
is the set of left-infinite
one-sided sequences in $(\Z/r\Z)^\N$ such that for every $m<n$ and
finite segment,
$w_{n-1}w_{n-2}\ldots w_m$, 
the sequence $(w_m,w_{m+1},\ldots w_{n-1})$ is the
sequence of labels 
of a path in the graph. To see this, notice that $(w_m,w_{m+1},\ldots,w_{n-1})$
is the sequence of labels of a path in the graph if and only if there
are elements $A_m,A_{m+1},\ldots,A_n$ of $V$ such that 
$\theta_{w_k}(A_{k+1})=A_k$ for each $m\le k<n$, so that
$|\theta_{w_m}\circ\theta_{w_{m+1}}\circ\cdots \theta_{w_{n-1}}(A_n)|
=|A_m|>c$ and hence 
$|\theta_{w_m}\circ\theta_{w_{m+1}}\circ\cdots \theta_{w_{n-1}}(\mathcal A)|>c$.

The reason for the reversal of the order of the
indices is that 
we want to maintain compatibility with standard definitions of sofic: 
$\Sigma_\theta$ is a collection of left-infinite sequences, which 
is precisely the collection of edge-labellings of
infinite paths in the graph \emph{listed in reverse order}.

\end{proof}

\begin{example}\label{example_0}\cite{coven_automorphism}.
If $|\mathcal A|=2$, then any primitive length $r$ substitution generating 
an infinite shift has height one, by Proposition \ref{prop:height}, 
and for each $i$, 
$\theta_i(\mathcal A)=\mathcal A$ or $|\theta_i(\mathcal A)|=1$. If $\mathcal C$ 
is the set of indices $i$ such that 
$|\theta_i(\mathcal A)|=1$, we see that $\mathcal F_\theta = \{
w=w_1 \ldots w_k \in (\Z/ r\Z)^+: w_i \in \mathcal C \text{ for some } 
i\in \{1, \ldots , k \} \}$. 
In this case $\Sigma_\theta$ is the full shift on the alphabet   
$(\Z/ r\Z) \backslash \mathcal C$. 
\end{example}

\begin{example}\label{example_one}
Let $\mathcal A = \{a,b,c \}$, and let $\theta (a)=abbc$, $\theta(b)=cbab$ and
$\theta(c)=cbba$; $\theta$  has a coincidence and is of  height one. The 
labelled graph that generates $\Sigma_\theta$
 is shown in Figure \ref{figure_1}.
\end{example}

\begin{figure}[h]
\begin{tikzpicture}[->,>=stealth',shorten >=1pt,auto,node distance=2.8cm,
                    semithick] 
 \tikzstyle{every state}=[fill=white,draw=black,text=black]
  \node[state]         (B)  {$\{ a,b,c\}$};
    \node[state]         (D) [below of=B]       {$\{a,b \}$};
     \node[state] (A)  [left of=D]                  {$\{a,c \}$};
 \node[state]         (C) [ right of=D] {$\{b,c\}$};
  \path (A) edge              node[pos=0.3] {0} (B)
     edge              node {0} (D)
        (B) edge [loop above] node{3} (B)
           (D) edge   [bend left=5]  node{2,3} (C)
                   (D) edge [loop below] node {2} (D)
              edge node{2} (B)
        (A) edge [loop left]  node {0,3} (A)
                      (C) edge    [bend left=5]      node{3} (D);
                      \end{tikzpicture}
\caption{\label{figure_1}The graph for Example \ref{example_one}.}
\end{figure}
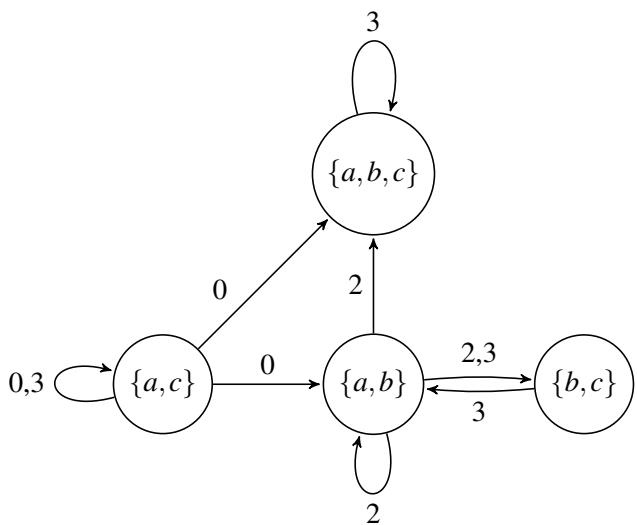

Depending on $\theta$, we may have to slightly alter the sofic 
shift that we work with. We write $\overline a$ to denote $\ldots aaa$. 

\begin{definition}\label{def_Y_F}

Let $\theta$ be a length $r$ substitution with column number $c$. 
Let $\Sigma_\theta$ be as above. 
Let $P=\{x\in X_\theta\colon\theta^n(x)=x\text{ for some $n$}\}$, that
is the set of (two-sided) periodic points under the substitution.
\begin{equation*}
\hat\Sigma_\theta :=
\begin{cases}
    \Sigma_\theta \cup \{ \overline 0, \overline {r-1} \} & 
    \text{ if $|P|>c$}\\
    \Sigma_\theta & \text{ otherwise;}
\end{cases}
\end{equation*} 
if $\Sigma_\theta$ is sofic, then  $\hat\Sigma_\theta$ is sofic. 
We shall show below (in Corollary \ref{cor:nonempty})
that $\hat\Sigma_\theta$ is not empty and hence by the definition 
above, $\hat\Sigma_\theta$ is always sofic. Define
\[ 
\tilde\Sigma_\theta= \{ x: \sigma^n(x) \in \hat\Sigma_\theta 
\mbox{ for some } n\geq 0\}. 
\]
\end{definition}

\begin{lemma}\label{lem:description_of_points_with_several_preimages}
Let $\theta$  be a primitive,  length $r$ substitution, with column number $c$, 
of height one and such that $(X_\theta,\sigma)$ is infinite. 
Then $\tilde\Sigma_\theta= \{ z\in \Z_r: |\pi^{-1}(z)|>c\}$.
\end{lemma}

\begin{proof}
We separate the proof into two parts: $z\in\Z$ and $z\not\in\Z$. 

First suppose $z\in\Z$ and $|\pi^{-1}(z)|>c$. Then there are more than $c$
$\theta$-periodic points, so that $z\in\Z\subset\tilde \Sigma_\theta$. 

For the converse, notice that if either 
$0^j\not\in\mathcal F_\theta$ for all $j$
or $(r-1)^j\not\in\mathcal F_\theta$ for all $j$, 
then there are more than $c$ 
one-sided right- or left- $\theta$-periodic points respectively. 
By primitivity, each one-sided $\theta$-periodic point 
extends to at least one two-sided $\theta$-periodic
point, so that we deduce that if $\Z\cap \tilde\Sigma_\theta$ is 
non-empty then there are more than $c$ $\theta$-periodic points. 
Hence if $z\in\Z\cap\tilde\Sigma_\theta$, we see $|\pi^{-1}(z)|>c$. 

If $z\in\tilde\Sigma_\theta\setminus\Z$, then there exists $n_0$ 
such that for all $n\ge n_0$, 
$z_n\ldots z_{n_0}\not\in\mathcal F_\theta$. We may assume 
without loss of generality 
(since $|\pi^{-1}(z)|=|\pi^{-1}(z+k)|$ for any $k\in\Z$) that 
$z_{n_0-1},\ldots,z_0$ are all 0.
Let $z'=\bar\sigma^{n_0}z$, so that 
for all  $n\geq0,$  $z'_{n-1} \ldots z'_0 \not \in \mathcal F_\theta$. 
This means that for each $n$ there are distinct letters 
$\bar a_n^{(1)}, \ldots  ,\bar a_n^{(c+1)}$, and 
letters $a_n^{(1)},\ldots ,a_n^{(c+1)}$ such that $\theta_{z'_0}\circ \ldots
\circ \theta_{z'_{n-1}} (a_n^{(i)}) = \bar a_n^{(i)}$ for $1\leq i \leq c+1$.
By passing to a subsequence we can assume that 
$(a_n^{(1)},\ldots , a_n^{(c+1)})=(a_1,\ldots ,a_{c+1})$  and 
$(\bar a_n^{(1)},\ldots , \bar a_n^{(c+1)})=
(\bar a_1,\ldots ,  \bar a_{c+1})$ for all $n$. 

For $i=1,\ldots,c+1$, take $x^{(n),i}\in X$ such that for each $i$,  
$x^{(n),i}_0=\bar a_i$, $\pi(x^{(n),i})$
ends  with $z'_{n-1}\ldots z'_0$ and the segment of $x^{(n),i}$ 
from coordinate index $-\sum_{i=0}^{n-1} z'_ir^i$ to 
$r^{n}-1-\sum_{i=0}^{n-1} z'_ir^i$
agrees with $\theta^n(a_i)$. Taking limits gives at 
least $c+1$ distinct points of $\pi^{-1}(z')$. 
Since $\theta$ is injective and $\theta^{n_0}x\in\pi^{-1}(z)$ 
for any $x\in\pi^{-1}(z')$,
applying $\theta^{n_0}$ to these points shows 
$|\pi^{-1}(z)|\ge c+1$.

If $z\not\in \tilde\Sigma_\theta\cup \Z$, then for infinitely 
many $n$,  there exist $k_n$ such that 
$z_{n+k_n} \ldots z_n\in \mathcal F_\theta$.
Let $x=(x_j)$ satisfy $\pi(x)=z$ and fix $n$. Since  
$|\theta_{z_{n}}\circ \ldots
\circ \theta_{z_{n+k_n}} (\mathcal A)| = c $,   
there are $c$ possible choices for the block of $x_j$'s  for 
$j \in [0, r^{n} -1] - \sum_{i=0}^{n-1}r^{i}z_i$.  
Since $z\not\in\Z$, the union of these intervals exhausts all of $\Z$. 
\end{proof}

\begin{corollary}\label{cor:nonempty}
Let $\theta$ be a primitive, length $r$ substitution defined 
on $\mathcal A$ with column number $c$, of height one, 
and such that  $(X_\theta, \sigma)$ is infinite. Then   
the function $|\pi^{-1}(z)|:\Z_r\rightarrow \N$ has 
minimum value $c$, is non-constant, and  
$\hat\Sigma_\theta$ is a nonempty and proper subset of $\Z_r$.
\end{corollary}

\begin{proof}
Let $z\in\Z_r$ and for each $a\in\mathcal A$, let $x^{(a)}\in X_\theta$
satisfy $x^{(a)}_0=a$. Let $y^{(n,a)}=\sigma^{m_n}\theta^n(x^{(a)})$,
where $m_n=z_{n-1}r^{n-1}+\ldots +z_1r+z_0$. Let $Y^{(n)}=\{y^{(n,a)}\colon 
a\in\mathcal A\}$ and let $Y^{(\infty)}$ be the set of accumulation points
of $Y^{(n)}$ as $n\to\infty$. The 0th coordinates of the 
$y^{(n,a)}$ exhaust $\theta_{z_0}\circ\theta_{z_1}\circ\cdots\theta_{z_{n-1}}
(\mathcal A)$, so that $|Y^{(n)}|\ge c$, and also $|Y^{(\infty)}|\ge c$.
Since $Y^{(\infty)}$ is a subset of $\pi^{-1}(z)$, we conclude 
$|\pi^{-1}(z)|\ge c$ for each $z\in\Z_r$. 

Since $\mathcal F_\theta \neq  \emptyset$, we take a word 
$w \in \mathcal F_\theta$, and then a point 
$z\in \Z_r\backslash \Z$ which contains 
$w$ infinitely often. Then $z\not \in \tilde\Sigma_\theta$, and
by Lemma \ref{lem:description_of_points_with_several_preimages}, 
$|\pi^{-1}(z)|=c$. 

To see that there exist points $z$ with $|\pi^{-1}(z)|> c$,    
we note that our shifts will always have at least one non-trivial 
right asymptotic orbit equivalence class. This follows from the fact 
that any  continuous, positively expansive map on an infinite compact 
metric space cannot be invertible \cite{Coven_Keane}.  We pick $x$ 
and $x'$ that are right asymptotic, and we suppose that 
$x_n=x'_n$  for $n\geq 0$. Let $z=\pi(x)=\pi(x')$. 
The sets $\theta_{z_0}\circ \theta_{z_1}\ldots \theta_{z_n}(\mathcal A)$ 
are decreasing in $n$ and so equal $\mathcal A_0$, a set of cardinality 
at least $c$, for all large $n$.
For each $a\in \mathcal A_0$,  there is a $y$ in $\pi^{-1}z$ with 
$y_0=a$. As $x$ and $x'$ are two elements in $\pi^{-1}z$ with the 
same 0 coordinate, we conclude that 
$|\pi^{-1}z|\geq c+1$.
\end{proof}

\subsection{Bounding denominators of $\kappa(\Phi)$} \label{subsec:denom_bound}

We have established that if $\theta$ has column number $c$, then  the set
$\{ z: |\pi^{-1}(z)|>c\}$ is the set of points whose tail
lies in  a sofic shift that is a proper subshift of the full
one-sided shift on $r$ letters, that is $\tilde\Sigma_\theta$.
Although we already know that $\kappa(\Phi)$ is rational for
$\Phi\in\mbox{Aut}(X_\theta,\sigma)$,  
Part (2) of  Theorem \ref{thm:ethan} tells us
that $\tilde\Sigma_\theta\subset\tilde\Sigma_\theta+\kappa(\Phi)$,
which will allow us to give bounds on the denominator of $\kappa(\Phi)$. 
We also recover a special case of a result in \cite{dhs}, namely 
that for our shifts, all endomorphisms  are  automorphisms.

We write $d$ for the $r$-adic metric: $d(x,y)=r^{-k}$,
where $k=\min\{j\colon x_j\ne y_j\}$ (or $d(x,y)=0$ if $x=y$).
We also need a metric on the circle. We identify the circle with 
$[0,1)$ and define $d_\circ(x,y)=\min_{n\in\Z}|n+(x-y)|$.

If $X$ is a subshift of $\Z_r$, we   say that $x$
has a \emph{tail} belonging to $X$ if there exists 
$k \in \N$ such that $\sigma^k(x)\in X$. 

\begin{lemma}\label{lem:soficsum}
Let $X$ be a proper sofic shift of $\Z_r$, so that there is
a word of length $j$ that does not occur in the language of $X$. 
Suppose that there exists $t\in \Z_r$ 
such that $x+t$ has a tail belonging to $X$ for each $x \in X$.  
Then $nt\in \Z$ for some $n$ satisfying $\gcd(n,r)=1$ and $n\le r^j-1$.
\end{lemma}

\begin{lemma}\label{lem:circdense}
Let $\Omega$ be an infinite subset of $S^1$. Then for each $\epsilon>0$, there
exists an $N$ such that 
$N\star\Omega:=\{n\omega\colon 1\le n\le N;\ \omega\in\Omega\}$ 
is $\epsilon$-dense.
\end{lemma}

\begin{proof}
If $\Omega$ contains an irrational point, then its multiples 
are dense and we're done.
Otherwise if $\Omega\subset\Q$, given $\epsilon$, since $\Omega$ is infinite,
it must contain a point $\frac pq$ (in lowest terms) with $q>1/\epsilon$. 
Now $\{j\cdot \frac pq\bmod 1: j\in \mathbb N_0\}$
exhausts $\{\frac iq\colon 0\le i<q\}$, so that $q\star\Omega$ 
is $\epsilon$-dense.
\end{proof}

\begin{proof}[Proof of Lemma \ref{lem:soficsum}]
Define $\psi_n\colon \Z_r\to S^1$  by 
$\psi_n(x)=\sum_{k=0}^{n-1}x_kr^{k-n}$. For $\omega\in S^1$, 
we denote $r\omega\bmod 1$ by $rs$.
Let $\Omega(t)\subset S^1$ be the set of limit points of 
$\{\psi_n(t): n\in \N_0\}$.
Notice that if $\psi_{n_i}(t)\to\omega$, then $\psi_{n_i-1}(t)\to r\omega$,
so that $\Omega(t)$ is closed under multiplication by $r$. Also $\Omega(mt)=
m\Omega(t)$  and $\Omega(m+t)= \Omega(t)$  for $m\in\Z$. 

We first claim that if $t$ is not rational, or if $t$ is rational with
a denominator exceeding $r^j$, then there exists an $m$ such that 
$\Omega(mt)$ contains an element of $S^1$ in $(r^{-j}/2,r^{-j})$. 

Suppose first that $\Omega(t)$ 
is a finite set. Since it is closed under multiplication
by $r$, it must consist of rationals (given $\omega\in\Omega(t)$, 
there exist $0\le m<n\le |\Omega(t)|$ such that
$r^m\omega$ and $r^n\omega$ agree modulo 1, so that $(r^n-r^m)\omega\in\Z$). 
In particular, by taking a least common multiple, 
there exists a $Q\in\N_0$ such that 
$Q\omega=0\bmod 1$ 
for all $\omega\in\Omega(t)$. 
Now as $d_\circ(\psi_n(t),\Omega(t))\to 0$, we have $d_\circ(\psi_n(Qt),0)=
d_\circ(Q\psi_n(t),Q\Omega(t))\to 0$. This implies $Qt\in\Z$, so 
that this is the rational case. By assumption, 
the denominator of $t$ exceeds $r^j$. 
By Lemma \ref{lem:divr}, we can write $t$ as $p/q$ with 
$\gcd(p,q)=\gcd(q,r)=1$. Then there exist  $a,m\in\Z$ such that  $mt=a-1/q$ 
in $\Z_r$. Now $\Omega(mt)=\Omega(-1/q)$. Since $q$ is coprime to $r$,
there exists a $b$ such that $r^b\equiv 1\pmod q$ (or $q|r^b-1$). Write 
$r^b-1=sq$ and notice that the expansion of $-1/q=s/(1-r^b)$ 
in $\Z_r$ equals $\bar s=\ldots sss$ where $\bar s$ has period $b$.
Now we see $\psi_{bn}(-1/q)\to s/(r^b-1)=1/q$, so that $1/q\in\Omega(mt)$. 
There exists an $l\in\Z$ such that $l/q\in (r^{-j}/2,r^{-j})$
In particular, we have shown the claim above (for $\Omega(lmt)$)
in the case that $t$ is rational.

If $\Omega(t)$ is infinite, then  by Lemma \ref{lem:circdense}, there
is an $M$ such that $\bigcup_{m=1}^M\Omega(mt)$ is 
$1/(2r^j+1)$-dense. Hence there is an $m\le M$ such that 
$\Omega(mt)\cap (r^{-j}/2,r^{-j})\ne\emptyset$ 
completing the proof of the claim. 

There is a natural bijection $\Phi$ between $j$-blocks of $X$ and $\Z/r^j\Z$:
given a $j$-block $B=[x_{j-1}\ldots x_0]$, $\Phi(B)=x_{j-1}r^{j-1}+\ldots+x_0$.
(Recall that we are writing elements of $\Z_r$ from right to left so that 
$x_{j-1}$ is the ``most significant digit'' of $B$).

Hence for a $j$-block
$B$, there is a natural notion of its successor $B+1$. Let $B$ be chosen so that 
$B$ is a valid $j$-block, but $B+1$ is not. To show that such a $B$ exists,
list the $j$-blocks sequentially, each one the successor of the previous one. 
At some point in the sequence there is a valid $j$-block 
followed by an invalid one 
(otherwise all would be valid or all would be invalid). 
Since $X$ is sofic, it has dense periodic points. Let $z$ be a periodic point
in $X$ containing $B$'s. Let $p$ be the period of $z$. 

By the above, there exist $\omega\in\Omega$ and $m$ such that 
$\frac 12 r^{-j}<m\omega\bmod 1<r^{-j}$. Since $\omega\in\Omega$, there exists 
an increasing sequence $(n_i)$ such that $\psi_{n_i}(t)\to \omega$. 
For all sufficiently large $i$, $\frac 12r^{-j}<\psi_{n_i}(mt)<r^{-j}$. 
Refine the sequence $(n_i)$ such that $\frac 12r^{-j}<\psi_{n_i}(mt)<r^{-j}$
for each $i$ and all of the $n_i$ are congruent modulo $p$ (such a
subsequence exists by the pigeonhole principle: at least one congruence class 
must contain infinitely many terms of the original sequence). 

Now let $k$ be chosen so that the block $B$ appears in the coordinate range
$n_i-j$ to $n_i-1$ for each $i$ in $w:=\sigma^k(z)$. Let $\alpha=\Phi(B)/r^j$.
The fact that $B$ appears in these blocks is equivalent to the assertion
that $\psi_{n_i}(w)\in [\alpha,\alpha+r^{-j})$, 
while the appearance of $B+1$ in locations 
$n-j$ to $n-1$ of $x\in X$ is equivalent to the assertion that 
$\psi_n(x)\in [\alpha+r^{-j},\alpha+2r^{-j})$. 

Now notice that for each $i$, either $\psi_{n_i}(w+mt)$ or
$\psi_{n_i}(w+2mt)$ lies in $[\alpha+r^{-j},\alpha+2r^{-j})$. 
In particular, one of $w+mt$ and $w+2mt$ contains infinitely many $B+1$ blocks,
and hence does not have a tail lying in $X$. 
Hence we have shown that if $t$ is irrational, 
or $t$ is rational with denominator
exceeding $r^j$, then the hypotheses of the theorem cannot be satisfied. 
\end{proof}

\begin{theorem}\label{thm:upper_bound}
Let $\theta$  be a primitive, length $r$ substitution on 
$\mathcal A$, with column number $c$, of height one,  
and such that $(X_\theta,\sigma)$ is infinite. 
Then $\mbox{End}(X_\theta,\sigma)=\mbox{Aut}(X_\theta,\sigma)$. 

If $\mathcal F_\theta$ contains a word of length $j$, then
any $\Phi\in\mbox{Aut}(X_\theta, \sigma)$ satisfies $\kappa(\Phi)=\ell/n$ 
for some $n\leq r^j-1$ with $\gcd(n,r)=1$, and some $\ell \in \Z$, and
$\Phi^{nk}=\sigma^{\ell k}$ for some $1\leq k\leq c$. \end{theorem}

\begin{proof}
Let $\theta$ be a substitution as in the statement and let $\Sigma_\theta$, 
$\hat\Sigma_\theta$ and $\tilde\Sigma_\theta$ be as constructed above. 
Let $\Phi$ be an endomorphism of $X_\theta$ and let $t=\kappa(\Phi)$. 
By Theorem \ref{thm:ethan}, Lemma 
\ref{lem:description_of_points_with_several_preimages} and
Corollary \ref{cor:nonempty}, we have $t+\tilde\Sigma_\theta\subset 
\tilde\Sigma_\theta$. By Lemma \ref{lem:soficsum}, we deduce $t$
is rational, $\ell /n$ say, with denominator coprime to $r$ and at most $r^j-1$.

Now $\kappa(\Phi^n\sigma^{-\ell })=0$, so that $\Phi^n\sigma^{-\ell }$
is a self-map of $\pi^{-1}(z)$ for any $z$ in $\Z_r$.
Choosing $z$ so that $|\pi^{-1}(z)|=c$, we see that there is a $1\le k\le c$
such that $(\Phi^n\sigma^{-\ell })^k$ has a fixed point. 
By minimality of $X_\theta$,
$\Phi^{nk}=\sigma^{\ell k}$ and we see that $\Phi$ is bijective, so
$\Phi\in\mbox{Aut}(X_\theta,\sigma)$. 
\end{proof}

\begin{example}\label{example_one_continued}
We continue with Example \ref{example_one}, which we already noted
has a coincidence and is of height one. Since the word 1 belongs to 
$\mathcal F_\theta$, Theorem \ref{thm:upper_bound} implies for any 
$\Phi\in\mbox{Aut}(X_\theta)$, $\kappa(\Phi)$ has denominator
1 or 3. If there were an automorphism with denominator 3, then by taking
powers and composing with a power of the shift, we could find
an automorphism $\Phi$ with $\kappa(\Phi)
=-1/3=\bar 1$. However, from the proof of Theorem \ref{thm:upper_bound}, this
would imply $\bar 1+\tilde{\Sigma_\theta}\subset\tilde{\Sigma_\theta}$, which
is false as $\bar 0\in\tilde{\Sigma_\theta}$, but $\bar 1\not\in
\tilde{\Sigma_\theta}$. Hence $\kappa(\Phi)\in\Z$ for all 
$\Phi\in\mbox{Aut}(X_\theta,\sigma)$ and such $\Phi$ are powers of the shift
by the theorem.
\end{example}

Recall that if $(X_\theta,\sigma)$ is a two-sided shift, then
$(\bar X_\theta,\bar\sigma)$ is the corresponding one-sided shift.

\begin{theorem}\label{thm:trivial_one_sided_group}
Let $\theta$  be a primitive, length $r$ substitution, 
of height one,  and such that $(X_\theta,\sigma)$ is infinite. 
If $\kappa$ is injective (in particular if $\theta$ has a coincidence), then
$\mbox{Aut}(\bar X_\theta,\bar \sigma)=\{\mbox{Id}\}$. 
\end{theorem}

\begin{proof}
Any automorphism $\bar \Phi$ of $(\bar X_\theta,\bar \sigma)$ 
gives rise to an endomorphism $\Phi$ of $(X_\theta,\sigma)$,  which satisfies
$\Phi^n=\sigma^k$ for some $k,n\in\N\cup\{0\}$. 
As $\bar \Phi$ is an
automorphism, we have  $k=0$ and so $\kappa (\Phi)=0$.
Now the injectivity of $\kappa$ implies that  $\Phi=\mbox{Id}$, so 
$\bar\Phi=\mbox{Id}$ also.
\end{proof}

\subsection{Computing $\mbox{Aut}(X_\theta,\sigma)$}\label{subsec:compaut}
In this section, we use the bounds on the denominator appearing in $\kappa$
together with bounds on the radius of the block code to obtain an 
algorithm to compute the automorphism group of a constant-length substitution.

Let $\Phi\in\mbox{Aut}(X_\theta,\sigma)$. By the Curtis--Hedlund--Lyndon theorem,
there exist $l$, $r$, and 
a map $f\colon \mathcal A^{l+r+1}\to \mathcal A$ with the property that
$(\Phi(x))_n= f(x_{n-l}, \ldots ,x_n, \ldots ,x_{n+r})$ for all $x$ and $n$.
Let $l,r$ be the smallest possible integers so that such an $f$ exists.
We call $l$ and $r$ 
the  left  and right radius of $\Phi$ respectively.   
We say $\Phi$ has  radius $R$ if
its left radius and right radius are both at most $R$.

The recognizability of $\theta$ implies that any  $x \in X_\theta$ 
can be written in a unique way as 
$x=\sigma^k(\theta(y))$ where $y\in X_\theta$ and, if $\theta$ is 
length $r$, $0\leq k<r$ (see \cite{host} for a proof in the 
injective case, which is all we need).
The following proposition tells us that up to a shift, 
every automorphism has a small radius,
parallel to the case 
of {\em reduced} constant-length substitutions with partial 
continuous spectrum in the work of Host and Parreau 
\cite[Theorem 1.3]{host_parreau}. 
We remark also that Coven, Dykstra, Keane and LeMasurier 
have a  similar result in \cite[Theorem 1]{CDKL}. 
  
\begin{proposition}\label{prop:radius_two} 
Let $\theta$  be an injective primitive, length $r$ substitution  of height one, 
and  such that $(X_\theta,\sigma)$ is infinite. 
If $\Phi\in \mbox{Aut} (X_\theta,\sigma)$ has the property
that $\kappa(\Phi)\in\Z_r\setminus \Z$ is periodic, then $\Phi$ 
has right radius zero and left radius at most 1. 

Further, if $\kappa(\Phi)=k/(1-r^p)$ for $0<k<r^p-1$, one has
\begin{equation*}
\theta^{-c!p }\circ \sigma^{-k(1+r^p+r^{2p}+\ldots + r^{(c!-1)p})}\circ 
\Phi\circ \theta^{c!p} = \Phi,
\end{equation*}
where $c$ is the column number.

If $\kappa(\Phi)=0$, then $\Phi$ has left and right radius at most 1 and 
$\Phi$ satisfies $\theta^{-c!}\circ\Phi\circ\theta^{c!}=\Phi$. 

\end{proposition}

\begin{proof}
We first assume that $\theta$ has a coincidence, so that $\kappa$ is an 
injection. Let $R$ be the radius of $\Phi$ and let $t=\kappa(\Phi)$. 
By assumption, $t$ is periodic as a point of $\Z_r$, with period $p$, say. 
Further, the repeating block is not all 0's or all $r-1$'s. 
Hence there exists $n$, a multiple of $p$, such that 
$t=r^nt+w$ and $R<w<r^n-R$. Now $(\Phi(x))_i$ is determined by $x_{[i-R,i+R]}$,
so that $(\sigma^{-w}\circ\Phi(x))_i$ is determined by $x_{[i-w-R,i-w+R]}$ and in
particular, $(\sigma^{-w}\circ\Phi(x))_{[0,r^n)}$ 
is determined by $x_{[-r^n,r^n)}$.

Since $\kappa(\sigma^{-w}\circ\Phi)\in r^n\Z_r$, $\sigma^{-w}\circ\Phi$ 
maps $\theta^n(X_\theta)$ to itself, and since recognizability tells 
us that $\theta^n : X_\theta \rightarrow \theta^n (X_\theta)$  is invertible,
let $\Psi(y):=\theta^{-n}\circ\sigma^{-w}\circ\Phi\circ
\theta^n(y)$. The above ensures that $\Psi$ is an 
automorphsim of $(X,\sigma)$, and the injectivity of $\theta$ implies
that $(\Psi(y))_0$ depends on $y_{[-1,0]}$. 
Notice that $\pi(\theta(x))=r\pi(x)$, so 
that $\pi(\Phi\circ\theta^nx)=r^n\pi(x)+\kappa(\Phi)$ and
$\pi(\Psi(x))=\pi(\theta^{-n}\circ\sigma^{-w}\circ\Phi\circ\theta^nx)=
r^{-n}(-w+r^n\pi(x)+\kappa(\Phi))=\pi(x) +t$. 
Hence $\kappa(\Psi)=t=\kappa(\Phi)$. 
By injectivity of $\kappa$ we see that $\Phi=\Psi$, so that $\Phi$ has
left radius 1 and right radius 0 as required. 

Suppose now that $\theta$ has column number $c$ and $t\in\kappa(\mbox{Aut}(X_\theta,T))
\setminus\Z$.
Theorem \ref{thm:ethan} 
tells us that the group homomorphism $\kappa$ has kernel of size $K\leq c$.
Let 
$\{\Phi_1, \ldots ,\Phi_K\}= \kappa^{-1}(t)$. As above, we find 
$n$ and $w$ that work for all of the $\Phi_i$'s, and consider  
the map which sends $\Phi \in \{ \Phi_1, \ldots ,\Phi_K\}$ to 
$\theta^{-n} \circ \sigma^{-w}\circ \Phi \circ \theta^n$. This 
map is a bijection of  $\{ \Phi_1, \ldots ,\Phi_K\}$, and, 
since $\theta^{-n} \circ \sigma^{-w}\circ \Phi \circ \theta^n$ 
has left radius one and right radius 0, we see that all the maps 
$\Phi_i$ satisfy this property. 

If $t=k/(1-r^p)$, the  calculation above shows that
$\Phi\mapsto \theta^{-p}\circ\sigma^{-k}\circ\Phi\circ\theta^p$ 
is a permutation of $\kappa^{-1}(t)$. 
Since $|\kappa^{-1}(t)|\le c$, one sees by iterating
\begin{equation*}
\theta^{-c!p}\circ \sigma^{-k(1+r^p+r^{2p}+\ldots + r^{(c!-1)p})}\circ 
\Phi \circ \theta^{c!p} = \Phi .\label{eq:commrel}
\end{equation*}

The properties of $\Phi$ such that $\kappa(\Phi)=0$ are established similarly. 
\end{proof}

If $t\in\Z_r$ is such that $nt=:k\in\Z$ for some $n$, 
then we define $\lfloor t\rfloor
=\lfloor k/n\rfloor \in \Z$. We define $\lceil t\rceil$ similarly as 
$\lceil k/n\rceil$. We can now rephrase the radius of $\Phi$ 
in terms of $\kappa(\Phi)$.

\begin{corollary}\label{cor:roundkappa}
Let $\theta$  be an injective, primitive, length $r$ substitution of height one, 
such that $(X_\theta,\sigma)$ is infinite. 
Then for each $\Phi\in\mbox{Aut}(X_\theta,\sigma)$ such that $\kappa(\Phi)\not\in\Z$,
 one has that
$\Phi(x)_0$ depends only on $x_{[\lfloor \kappa(\Phi)\rfloor, 
\lceil\kappa(\Phi)\rceil]}$. 
\end{corollary}

\begin{proof}
Let $n$ be chosen so that $\kappa(\Phi)-n$ is periodic in $\Z_r$ with some 
period $p$, and let $\Psi=\Phi\circ\sigma^{-n}$. If the periodic block is the
base $r$ representation of a number $a$, we have 
$\kappa(\Psi)=a/(1-r^p)$ with $0<a<r^p-1$, so that $(r^p-1)\kappa(\Psi)=-a$. 
Hence  $\kappa(\Phi)=n-a/(r^p-1)$. 
In particular, $\lfloor\kappa(\Phi)\rfloor=n-1$ and $\lceil\kappa(\Phi)\rceil=n$.
By Proposition \ref{prop:radius_two}, $\Psi(x)_0$ depends on $x_{[-1,0]}$, and 
the result follows.
\end{proof}

We remark that a version of Corollary
\ref{cor:roundkappa} exists for integral values of $\kappa(\Phi)$ also. 
If $\theta$ has a coincidence, then integral values of $\kappa(\Phi)$ 
only arise from powers of the shift. Otherwise, 
if $\kappa(\Phi)=n$, then $\Phi(x)_0$ is determined by 
$x_{n-1}$, $x_n$ and $x_{n+1}$. 
If $\theta$ has the property that for each pair $a$ and $a'$ of
distinct elements of $\mathcal A$, there is a $k$ such that $\theta^k(a)$ and
$\theta^k(a')$ differ in some coordinate that is neither the first
nor the last of the block, then
one can prove that $\Phi(x)_0$ is determined by $x_n$. 
A natural question is whether this is the case for all 
primitive constant-length substitutions.

Given left and right radii $m$ and $n$ and a map 
$f\colon \mathcal A^{m+n+1}\to \mathcal A$,
we write $\Phi^f_{m,n}$ for the map $\mathcal A^\Z\to \mathcal A^\Z$ 
given by $\Phi^f_{m,n}(x)_i=f(x_{i-m},\ldots,x_{i+n})$.

\begin{proposition}\label{prop:last_nail_lemma}
Let $\theta$  be an injective primitive, length $r$ substitution on 
$\mathcal A$,  with column number $c$, of height one, and  
such that $(X_\theta,\sigma)$ 
is infinite. 
Let $p\ge 1$ and $0<k<r^p-1$. 

Let $f\colon \mathcal A^2\to \mathcal A$.
Then $\Phi:=\Phi^f_{1,0}$ satisfies $\Phi\in \mbox{Aut}(X_\theta,\sigma)$  
and $\kappa(\Phi)=k/(1-r^p)$ if and only if:
\begin{enumerate}
\item 
if $x_0x_1x_2\in\mathcal L_{X_\theta}$, then $f(x_0,x_1)f(x_1,x_2)\in
\mathcal L_{X_\theta}$; and \label{cond:fblock}
\item For each $x_{-1}x_0x_1\in\mathcal L(X_\theta)$ and $0\le i<r^{c!p}$
\label{cond:crelf}
$$
f(\theta^{c!p}(x_{-1}\cdot x_0)_{i-1},\theta^{c!p}(x_{-1}\cdot x_0)_i)
=\theta^{c!p}(f(x_{-1},x_0)f(x_0,x_1))_{N+i},
$$
where $N=k(1+r^p+\ldots+r^{(c!-1)p})$ and $x_{-1}\cdot x_0$ denotes a 
finite segment of a bi-infinite sequence with $x_{-1}$ in the $-1$st coordinate
and $x_0$ in the 0th coordinate.
\end{enumerate}

Similarly given $g\colon \mathcal A^3\to \mathcal A$, $\Phi^g_{1,1}$ 
satisfies $\Phi^g_{1,1}\in\mbox{Aut}(X_\theta,\sigma)$
and $\kappa(\Phi^g_{1,1})=0$ if and only if 
\begin{enumerate}
\setcounter{enumi}{2}
\item $g(x_0,x_1,x_2)g(x_1,x_2,x_3)\in\mathcal L_{X_\theta}$
whenever $x_0x_1x_2x_3\in\mathcal L_{X_\theta}$; and
\label{cond:gblock}
\item For any $x_{-1}x_0x_1\in\mathcal L_{X_\theta}$ and $0\le i<r^{c!}$, 
\label{cond:crelg}
\begin{equation*}
g(u_{i-1},u_i,u_{i+1})=
(\theta^{c!}g(x_{-1},x_0,x_1))_i \, ,
\end{equation*}
where $u=u_{-r^{c!}}u_{-r^{c!}+1}\ldots u_{-1}\cdot u_0\ldots
u_{2r^{c!}-1}$ is the word 
obtained by applying $\theta^{c!}$ to the word $x_{-1}\cdot x_0x_1$.
\end{enumerate}
\end{proposition}

\begin{proof}
First suppose that there exists $\Phi\in\mbox{Aut}(X_\theta,\sigma)$ such that
$\kappa(\Phi)=t:=k/(1-r^p)$. By Proposition \ref{prop:radius_two}, 
$\Phi=\Phi^f_{1,0}$ for a map $f\colon \mathcal A^2\to \mathcal A$. 
Condition \eqref{cond:fblock} is obviously
satisfied and condition \eqref{cond:crelf} is the commutation relation of 
Proposition \ref{prop:radius_two} expressed in terms of $f$.

Conversely, suppose $f$ has satisfies Conditions
\eqref{cond:fblock} and \eqref{cond:crelf}. Let $\Phi=\Phi^f_{1,0}$. 
Condition \eqref{cond:crelf} implies 
$\Phi(\theta^{c!p}x)_i=(\sigma^N\circ \theta^{c!p}\circ \Phi(x))_i$
for $0\le i<r^{c!p}$. The same equation applied to iterates of $x$ under $\sigma$
yields $\Phi\circ\theta^{c!p}=\sigma^N\circ\theta^{c!p}\circ\Phi$.
Iterating this equality and using the relation 
$\sigma^r\circ\theta=\theta\circ\sigma$ gives 
\begin{equation*}
\Phi\circ\theta^{nc!p}=(\sigma^N\circ \theta^{c!p})^n\circ\Phi=
\sigma^{N(1+r^{c!p}+\ldots+r^{(n-1)c!p})}\circ \theta^{nc!p}\circ\Phi.
\end{equation*} 

We let $x\in X_\theta$ and
show that $\Phi(x)\in X_\theta$. For any $n$, we can write
$x=\sigma^{-j}\circ \theta^{npc!}(y)$ for some 
$y\in X_\theta$ and $0\le j<r^{npc!}$. Now applying the above, we have
\begin{equation*}
\Phi(x)=\sigma^{-j}\circ \Phi(\theta^{npc!}(y))=
\sigma^{l}(\theta^{npc!}(\Phi(y))),
\end{equation*}
where 
\[
l=-j+N(1+r^{c!p}+\ldots+r^{(n-1)c!p}).
\] 
Since by \eqref{cond:fblock} all 2-words of $\Phi(y)$ belong to 
$\mathcal L_{X_\theta}$, we deduce all subwords of 
$\theta^{npc!}(\Phi(y))$ of size $1+r^{npc!}$ belong to 
$\mathcal L_{X_\theta}$. Since $n$ is arbitrary, we deduce $\Phi(x)\in X_\theta$. 

Since $\Phi$ is shift-commuting by construction, we have 
$\Phi\in\mbox{End}(X_\theta,\sigma)$ and hence $\Phi\in\mbox{Aut}(X_\theta,\sigma)$
by Theorem \ref{thm:upper_bound}.
In particular, $t:=\kappa(\Phi)$ is now defined. 
Now
$\pi(\Phi(\theta^{pc!}(x)))=t+r^{pc!}\pi(x)$
and $\pi(\sigma^N  
(\theta^{pc!}(\Phi(x))))=k   (1+r^p+ \ldots +r^{(c!-1)p}    )      
+r^{pc!}(\pi(x)+t)$.
Since these agree, we have $t(1-r^p)=k$ as required. 

The argument that Conditions \eqref{cond:gblock} and \eqref{cond:crelg}
characterize elements of $\ker\kappa$ is similar.
\end{proof}

We can now deduce

\begin{theorem}\label{thm:last_nail}
Let $\theta$  be a primitive, length $r$ substitution on 
$\mathcal A$  such that $(X_\theta,\sigma)$ is infinite.
Then there is an algorithm to compute $\mbox{Aut}(X_\theta,\sigma)$.
\end{theorem}

\begin{proof}
If $\theta$ does not have height one, we use Proposition \ref{prop:height}
to compute a pure base of $\theta$. Proposition 
\ref{prop:automorphism_with_height} tells us how to retrieve 
$\mbox{Aut}(X_\theta,\sigma)$ from the automorphism group of its pure base. 
Similarly, if the pure base is not injective, we use 
Theorem \ref{thm:injective} to compute its automorphism group 
from the automorphism group of a conjugate injective substitution shift.
Hence it suffices to give an algorithm for injective substitutions of height one.

We can algorithmically compute the column number, $c$ and also $j$, the
length of the shortest word in $\mathcal F_\theta$ from the transition
graphs described in Section \ref{subsec:transgraph}. 
Recall from Lemma \ref{lem:cyclic} that
$\kappa(\mbox{Aut}(X_\theta,T))$ is a cyclic subgroup of $\Z_r$ 
generated by $1/d$ for some integer $d$.
By Theorem \ref{thm:upper_bound}, $d$ satisfies $\gcd(d,r)=1$ and
$d\le r^j-1$. 
Now $\mbox{Aut}(X_\theta,T)$ is the 
semi-direct product of $\ker\kappa$ with the cyclic subgroup generated
by any element of $\kappa^{-1}(-1/d)$. 

To find $\ker\kappa$, one applies Proposition \ref{prop:last_nail_lemma}
to test all of the finitely many 3-block maps. Note that one can 
algorithmically list the words of length two and 
three that belong to $\mathcal L_{X_\theta}$. One then needs 
to find an element of $\kappa^{-1}(-1/d)$ for the largest $d$ so that this
set is non-empty. By the above, there are finitely
many values of $d$ to check. Since $\gcd(d,r)=1$, we see $d|r^p-1$ for some $p$
and then $-1/d$ can be expressed as $k/(1-r^p)$ for some $0<k<r^p -1$.
By Proposition \ref{prop:last_nail_lemma},
there are finitely many 2-block maps to check for each potential $d$. 
\end{proof}

We do not  address the question of efficiency.

From Corollary \ref{cor:cyclic}, we know that for 
substitutions with a coincidence,
$\mbox{Aut}(X_\theta,\sigma)/\langle \sigma\rangle$
is a finite cyclic group. In Example \ref{example_one_continued}, 
this group was trivial, and 
it is natural to wonder if this is always the case. The following example shows 
that the quotient can indeed be non-trivial.

\begin{example}\label{example_three} Let
 $\mathcal A = \{a,b,c \}$,  $\theta(a)=aba$, $\theta(b)=cba$ and 
$\theta(c)=ccb$. This example appears in Dekking's article \cite{dekking} 
just after Theorem 14. 
He describes how $(X_\theta,\sigma)$ is a relabelling of the second
power shift of the substitution 
$\theta'$ where $\theta'(0)=011$ and $\theta'(1) = 101$. That is, the shift
with symbols $\scalebox{0.9}{\fbox{01}}$, $\scalebox{0.9}{\fbox{11}}$ 
and \scalebox{0.9}{$\fbox{10}$} with 
$\theta(\,\scalebox{0.9}{\fbox{01}}\,)=
\scalebox{0.9}{\fbox{01}\,\fbox{11}\,\fbox{01}}$, 
$\theta(\,\scalebox{0.9}{\fbox{11}}\,)=
\scalebox{0.9}{\fbox{10}\,\fbox{11}\,\fbox{01}}$
and $\theta(\,\scalebox{0.9}{\fbox{10}}\,)=
\scalebox{0.9}{\fbox{10}\,\fbox{10}\,\fbox{11}}$.
Then $(X_\theta, \sigma)$ 
is conjugate to $(X_{\theta'}, \sigma'^2)$ (where we use
$\sigma'$ to denote the shift on $X_{\theta'}$). 
As a consequence, the map $\sigma'$ is conjugate to an automorphism $\Phi$
of $(X_\theta,\theta)$ satisfying $\Phi^2=\sigma$. 
Using the techniques we 
have developed, it can be verified that $t=\overline 1$ is the only 
non-trivial periodic 3-adic integer  
such that $\tilde\Sigma_\theta +t = \tilde\Sigma_\theta$.
Since $\Phi^2=\sigma$, we have $\kappa(\Phi)= \overline1 2$ and 
we have shown that $\mbox{Aut}(X_\theta,\sigma)=\langle \Phi\rangle$. 

\end{example}

\subsection{Computing $\mbox{Conj}(X_\theta,X_{\theta'})$.}
\label{subsec:conjugacy_computation}
The techniques of Section \ref{subsec:compaut} can be generalized 
to compute the set of conjugacies between two substitution shifts 
generated by primitive constant-length substitutions,  
one of which generates an infinite shift. We use
$\kappa$ to also denote the  restriction to $\mbox{Conj}(X_\theta,X_{\theta'})$ 
of the map in Theorem \ref{thm:ethan}. We continue to assume that our maximal 
equicontinuous factor mappings are those that were defined in 
Section \ref{subsec:transgraph}, so that $r^n|\pi(x)$ if and 
only if $x\in\theta^n(X_\theta)$.

\begin{proposition}\label{prop:upper_bound_conjugacy} 
Let $\theta$  and $\theta'$ be  primitive,  length $r$ substitutions, 
both  with column number $c$ and of height one,  and 
such that $(X_\theta,\sigma)$ and  
$(X_{\theta'},\sigma)$ are infinite.  
Suppose that $\Phi\in \mbox{Conj}(X_\theta,X_{\theta'})$. Then  
$\kappa (\Phi)$ is  rational, and if $\mathcal F_\theta$ contains a  
word of length $j$, then $n\kappa (\Phi) \in \Z$ for some 
$0\leq n\leq (r-1)( r^{j}-1)$.
\end{proposition}

\begin{proof} Let $\pi:X_\theta \rightarrow \Z_r$ and 
$\pi':X_{\theta'} \rightarrow \Z_r$ denote the relevant  
maximal equicontinuous factor mappings, and suppose that 
$\Phi \in \mbox{Conj}(X_\theta, X_{\theta'}).$ Since 
$\kappa (\Phi \circ \sigma^n)= 
\kappa (\Phi) + n$, we  can assume, 
by composing $\Phi$ with a power of the shift if necessary, 
that $\kappa (\Phi) \equiv 0 \mod r$, i.e. that
$\kappa(\Phi)\in r\Z_r$.
In this case, noting that $\pi(\theta(x))=r\pi(x)$, we have 
$\pi'(\Phi(\theta(x)))= \kappa (\Phi)+ 
\pi(\theta(x)) =  \kappa (\Phi)+ r\pi(x) \equiv 0\mod r$.   
Since recognizability tells us that 
$\theta: X_\theta \rightarrow \theta (X_\theta)$  and 
$\theta': X_{\theta'} \rightarrow {\theta'} (X_{\theta'})$  
are invertible, then $\Psi:=  \theta'^{-1} \circ\Phi \circ \theta$ 
is a well defined conjugacy. Also,
\[ 
\pi(x) + \kappa (\Psi) = 
\pi'( \theta'^{-1} \circ\Phi \circ \theta  (x))= 
\frac{1}{r} (\pi(\theta(x))+\kappa (\Phi)) =
\pi(x) +\frac{\kappa (\Phi) }{r},  
\]
so that $\kappa  (\Psi)=
\kappa (\Phi)/r$. Finally,  (1) of 
Theorem \ref{thm:ethan} tells  us that 
\[
\kappa (\Psi\circ \Phi^{-1})= \kappa (\Psi) - 
\kappa (\Phi)= \kappa (\Phi)(\tfrac{1}{r}-1),
\]
and  since $\Psi\circ \Phi^{-1} \in \mbox{Aut}(X_{\theta'}, \sigma)$, 
so that $\kappa (\Psi\circ \Phi^{-1})$   is rational,  
we deduce our first claim. 
By Theorem \ref{thm:upper_bound},
$n\kappa (\Psi\circ \Phi^{-1})\in \Z$ for some $0\leq n\leq r^j-1$, 
and our second claim follows.
\end{proof}

The proofs of the following two propositions are a straightforward 
generalization  of those of  Propositions \ref{prop:radius_two} 
and \ref{prop:last_nail_lemma}.

\begin{proposition}\label{prop:radius_two_conjugacy} 
Let $\theta$  and $\theta'$ be  injective primitive,  length $r$ substitutions, 
both  with column number $c$ and of  height one,   and 
such that $(X_\theta,\sigma)$ and  
$(X_{\theta'},\sigma)$ are infinite.  
If $\Phi\in \mbox{Conj}(X_\theta, X_{\theta'})$ has the property
that $\kappa (\Phi)\in\Z_r\setminus\Z$ is periodic, then $\Phi$ 
has right radius zero and left radius at most $1$. 
Further, one has
\begin{equation*}
\theta'^{-c! p}\circ \sigma^{-k(1+r^p+r^{2p}+\ldots + r^{(c!-1)p})}\circ 
\Phi\circ \theta^{c!p} = \Phi,
\end{equation*}
where $c$ is the column number and $\kappa(\Phi)=k/(1-r^p)$
with $0<k<r^p-1$.

If $\kappa (\Phi)=0$, then $\Phi$ 
has left and right radius  at most $1$, and 
\begin{equation*}
\theta'^{-c! }\circ 
\Phi\circ \theta^{c!} = \Phi.
\end{equation*}
\end{proposition}

\begin{proposition}\label{prop:last_nail_conjugacy_lemma}
Let $\theta$  and $\theta'$ be injective primitive, length $r$ substitutions  
on $\mathcal A$ and $\mathcal A'$ respectively, both with column 
number $c$ and of  height one, and such that $(X_\theta,\sigma)$ and  
$(X_{\theta'},\sigma)$ are infinite. Let $p\ge 1$ and suppose that 
$0<k<r^p-1$.

Let $f\colon \mathcal A^2\to \mathcal A'$.
Then $\Phi:=\Phi^f_{1,0}$ satisfies $\Phi\in \mbox{Conj}(X_\theta,X_{\theta'})$  
and $\kappa(\Phi)=k/(1-r^p)$ if and only if:
\begin{enumerate}
\item 
if $x_0x_1x_2\in\mathcal L_{X_\theta}$, then $f(x_0,x_1)f(x_1,x_2)\in
\mathcal L_{X_{\theta'}}$; and \label{cond:fblockconj}
\item For each $x_{-1}x_0x_1\in\mathcal L(X_\theta)$ and $0\le i<r^{c!p}$
\label{cond:crelfconj}
$$
f(\theta^{c!p}(x_{-1}\cdot x_0)_{i-1},\theta^{c!p}(x_{-1}\cdot x_0)_i)
=\theta'^{c!p}(f(x_{-1},x_0)f(x_0,x_1))_{N+i},
$$
where $N=k(1+r^p+\ldots+r^{(c!-1)p})$ and $x_{-1}\cdot x_0$ denotes a 
finite segment of a bi-infinite sequence with $x_{-1}$ in the $-1$st coordinate
and $x_0$ in the 0th coordinate.
\end{enumerate}

Similarly given $g\colon \mathcal A^3\to \mathcal A'$, $\Phi^g_{1,1}$ 
satisfies $\Phi^g_{1,1}\in\mbox{Conj}(X_\theta,X_{\theta'})$
and $\kappa(\Phi^g_{1,1})=0$ if and only if 
\begin{enumerate}
\setcounter{enumi}{2}
\item $g(x_0,x_1,x_2)g(x_1,x_2,x_3)\in\mathcal L_{X_{\theta'}}$
whenever $x_0x_1x_2x_3\in\mathcal L_{X_\theta}$; and
\label{cond:gblockconj}
\item For any $x_{-1}x_0x_1\in\mathcal L_{X_\theta}$ and $0\le i<r^{c!}$, 
\label{cond:crelgconj}
\begin{equation*}
g(u_{i-1},u_i,u_{i+1})=
(\theta'^{c!}g(x_{-1},x_0,x_1))_i  \,,
\end{equation*}
where $u$ is the block of length $3r^{c!}$
given by $u_{-r^{c!}}u_{-r^{c!}+1}\ldots u_{-1}\cdot u_0\ldots u_{2r^{c!}-1}
=\theta^{c!}(x_{-1}\cdot x_0x_1)$.
\end{enumerate}
\end{proposition}

Let $\theta'$ be a primitive length $r$ substitution with column number $c$,
such that   $(X_{\theta'}, \sigma)$ is infinite. 
Let  $\theta$ be 
another  constant-length substitution.  If $(X_\theta, \sigma)$ and 
$(X_{\theta'},\sigma)$ were topologically conjugate,  
then both systems would be infinite,   and both would have the same 
maximal equicontinuous factor. 
Thus,  conjugacy of the two systems 
implies that $\theta$ and $\theta'$ have the same height, 
and the appropriate dynamical formulation of Cobham's theorem \cite{Coven_Cobham} 
tells us that the lengths of $\theta$ and $\theta'$ 
are powers of the same integer. Corollary \ref{cor:nonempty} 
implies that both have column number $c$. By taking 
a power of $\theta'$ or $\theta$ if necessary, and by considering the 
pure bases of our two substitutions, we are in a position to apply
Propositions \ref{prop:upper_bound_conjugacy},
\ref{prop:radius_two_conjugacy} and \ref{prop:last_nail_conjugacy_lemma}, 
to conclude:
 
\begin{theorem}\label{thm:last_nail_conjugacy_theorem}
Let $\theta$ and $\theta'$ be  primitive, constant-length substitutions. 
Suppose further $(X_{\theta}, \sigma)$ is infinite.
Then there is an algorithm to decide whether $(  X_{\theta}, \sigma  )$ 
and $(X_{\theta'}, \sigma)$ are topologically conjugate. In the case 
that they are topologically conjugate, the algorithm yields a 
topological conjugacy and hence one can algorithmically determine 
$\mbox{Conj}(X_\theta, X_{\theta'})$. 
\end{theorem}


\section*{Acknowledgments} 
We thank Michel Dekking, Bryna Kra, Alejandro Maass, 
Samuel Petite and Marcus Pivato for helpful comments. 
We also thank the referee for a thorough reading and many helpful comments.
The third author thanks LIAFA, Universit\'e Paris-7 for 
its hospitality and support.

\bibliographystyle{amsplain}

\begin{thebibliography}{99}

\bibitem{bdh}
M.~Barge, B.~Diamond, and C.~Holton.
\newblock Asymptotic orbits of primitive substitutions.
\newblock {\em Theoret. Comput. Sci.}, 301(1-3):439--450, 2003.

\bibitem{cas}
J.~Cassaigne.
\newblock Special factors of sequences with linear subword complexity.
\newblock In {\em Developments in language theory, {II} ({M}agdeburg, 1995)},
  pages 25--34. World Sci. Publ., River Edge, NJ, 1996.

\bibitem{coven_automorphism}
E.~M. Coven.
\newblock Endomorphisms of substitution minimal sets.
\newblock {\em Z. Wahrscheinlichkeitstheorie und Verw. Gebiete}, 20:129--133,
  1971/72.

\bibitem{CDK}
E.~M. Coven, M.~Dekking, and M.~Keane.
\newblock Topological conjugacy of constant length substitution dynamical systems.
\newblock {\em  Indag. Math. (N.S.)},  28  (2017),   no. 1, 91--107.
		


\bibitem{CDKL}
E.~M. Coven, A.~Dykstra, M.~Keane, and M.~LeMasurier.
\newblock Topological conjugacy to given constant length substitution minimal
  systems.
\newblock {\em Indag. Math. (N.S.)}, 25(4):646--651, 2014.

\bibitem{Coven_Cobham}
E.~M. Coven, A.~Dykstra, and M.~Lemasurier.
\newblock A short proof of a theorem of {C}obham on substitutions.
\newblock {\em Rocky Mountain J. Math.}, 44(1):19--22, 2014.

\bibitem{Coven_Keane}
E.~M. Coven and M.~Keane.
\newblock Every compact metric space that supports a positively expansive
  homeomorphism is finite.
\newblock In {\em Dynamics \& stochastics}, volume~48 of {\em IMS Lecture Notes
  Monogr. Ser.}, pages 304--305. Inst. Math. Statist., Beachwood, OH, 2006.

\bibitem{cyr_kra_2}
V.~Cyr and B.~Kra.
\newblock The automorphism group of a shift of linear growth: beyond
  transitivity.
\newblock {\em Forum Math. Sigma}, 3:e5, 27, 2015.

\bibitem{dekking}
F.~M. Dekking.
\newblock The spectrum of dynamical systems arising from substitutions of
  constant length.
\newblock {\em Z. Wahrscheinlichkeitstheorie und Verw. Gebiete},
  41(3):221--239, 1977/78.

\bibitem{durand_petite_maass}
S.~Donoso, F.~Durand, A.~Maass, and S.~Petite.
\newblock On automorphism groups of low complexity minimal subshifts.
\newblock {\em Ergodic Theory Dynam. Systems}, 36(1):64--95, 2016.

\bibitem{downarowicz}
T.~Downarowicz.
\newblock Survey of odometers and {T}oeplitz flows.
\newblock In {\em Algebraic and topological dynamics}, volume 385 of {\em
  Contemp. Math.}, pages 7--37. Amer. Math. Soc., Providence, RI, 2005.

\bibitem{fdlr}
F.~Durand.
\newblock Linearly recurrent subshifts have a finite number of non-periodic
  subshift factors.
\newblock {\em Ergodic Theory Dynam. Systems}, 20(4):1061--1078, 2000.

\bibitem{dhs}
F.~Durand, B.~Host, and C.~Skau.
\newblock Substitutional dynamical systems, {B}ratteli diagrams and dimension
  groups.
\newblock {\em Ergodic Theory Dynam. Systems}, 19(4):953--993, 1999.

\bibitem{DL}
F.~Durand and J.~Leroy.
\newblock Personal communication.

\bibitem{BDM}
F.~D. F.~Blanchard and A.~Maass.
\newblock {\em Nonlinearity}, 17:817--833, 2005.

\bibitem{host}
B.~Host.
\newblock Valeurs propres des syst\`emes dynamiques d\'efinis par des
  substitutions de longueur variable.
\newblock {\em Ergodic Theory Dynam. Systems}, 6(4):529--540, 1986.

\bibitem{host_parreau}
B.~Host and F.~Parreau.
\newblock Homomorphismes entre syst\`emes dynamiques d\'efinis par
  substitutions.
\newblock {\em Ergodic Theory Dynam. Systems}, 9(3):469--477, 1989.

\bibitem{kamae}
T.~Kamae.
\newblock A topological invariant of substitution minimal sets.
\newblock {\em J. Math. Soc. Japan}, 24:285--306, 1972.

\bibitem{lemanczyk_mentzen}
M.~Lema{\'n}czyk and M.~K. Mentzen.
\newblock On metric properties of substitutions.
\newblock {\em Compositio Math.}, 65(3):241--263, 1988.

\bibitem{martin}
J.~C. Martin.
\newblock Substitution minimal flows.
\newblock {\em Amer. J. Math.}, 93:503--526, 1971.

\bibitem{michel}
P.~Michel.
\newblock Stricte ergodicit\'e d'ensembles minimaux de substitution.
\newblock {\em C. R. Acad. Sci. Paris S\'er. A}, 278:811--813, 1974.

\bibitem{mosse1}
B.~Moss{\'e}.
\newblock Puissances de mots et reconnaissabilit\'e des points fixes d'une
  substitution.
\newblock {\em Theoret. Comput. Sci.}, 99(2):327--334, 1992.

\bibitem{olli}
J.~Olli.
\newblock Endomorphisms of {S}turmian systems and the discrete chair
  substitution tiling system.
\newblock {\em Discrete Contin. Dyn. Syst.}, 33(9):4173--4186, 2013.

\bibitem{salo_torma}
V.~Salo and I.~T{\"o}rm{\"a}.
\newblock Block maps between primitive uniform and {P}isot substitutions.
\newblock {\em Ergodic Theory Dynam. Systems}, 35(7):2292--2310, 2015.

\end{thebibliography}
\def\ocirc#1{\ifmmode\setbox0=\hbox{$#1$}\dimen0=\ht0 \advance\dimen0
  by1pt\rlap{\hbox to\wd0{\hss\raise\dimen0
  \hbox{\hskip.2em$\scriptscriptstyle\circ$}\hss}}#1\else {\accent"17 #1}\fi}


\begin{dajauthors}
\begin{authorinfo}[ethan]
  Ethan M. Coven\\
     Department of Mathematics\\ Wesleyan University \\
  U.S.A.\\
  ecoven\imageat{}wesleyan\imagedot{}edu \\
\end{authorinfo}
\begin{authorinfo}[anthony]
  Anthony Quas\\
  Department of Mathematics and Statistics \\University of Victoria 
\\
  Canada\\
  aquas\imageat{}uvic\imagedot{}ca \\
\end{authorinfo}
\begin{authorinfo}[reem]
  Reem Yassawi\\
Department of Mathematics\\ Trent University\\
  Canada\\
  ryassawi\imageat{}trentu\imagedot{}ca\\
\end{authorinfo}
\end{dajauthors}

\end{document}